\documentclass[a4paper,10pt]{article}
\usepackage[utf8]{inputenc}
\usepackage{amssymb}
\usepackage{amsmath}
\usepackage{amsthm}
\usepackage[top=1in, bottom=1.25in, left=1.25in, right=1.25in]{geometry}
\usepackage{comment}
\usepackage{graphicx}
\usepackage[disable]{todonotes}
\usepackage{stmaryrd}
\usepackage{geometry}
\usepackage{pdfpages}
\usepackage{comment}
\usepackage{tikz-cd}
\usepackage{float}
\usepackage{cleveref}
\usepackage[backend=biber,maxbibnames=99]{biblatex}
\numberwithin{equation}{section}
\newcommand{\dd}{\ensuremath{\mathrm{d}}}

\newcommand{\ver}{\ensuremath{\mathcal{V}}}

\newcommand{\spp}{\ensuremath{\mathrm{Sp}(2)}}

\newcommand{\cp}{\ensuremath{\mathbb{CP}^3}}
\newcommand{\Z}{\ensuremath{\mathbb{Z}}}
\newcommand{\R}{\ensuremath{\mathbb{R}}}
\newcommand{\C}{\ensuremath{\mathbb{C}}}
\newcommand{\HH}{\ensuremath{\mathbb{H}}}
\newcommand{\uo}{\ensuremath{\mathrm{U}(1)}}
\newcommand{\y}{\ensuremath{|Z|}}

\newcommand{\nk}{\ensuremath{\overline{\nabla}}}
\newcommand{\sff}{\ensuremath{\mathrm{I\!I}}}
\newcommand{\tcp}{\ensuremath{_{\mathbb{T}^2}\backslash^{\mathbb{CP}^3}}}
\newcommand{\taction}{\ensuremath{_{\mathbb{T}^2}\backslash^}}
\newtheorem{thm}{Theorem}
\numberwithin{thm}{section}
\newtheorem*{thm*}{Theorem}
\newtheorem{lma}[thm]{Lemma}
\theoremstyle{remark}
\newtheorem{rmk}[thm]{Remark}
\theoremstyle{remark}
\newtheorem*{rmk*}{Remark}
\newtheorem{crl}[thm]{Corollary}
\theoremstyle{definition}

\newtheorem*{dfn*}{Definition}
\newtheorem{prop}[thm]{Proposition}
\newtheorem{prop*}{Proposition*}
\theoremstyle{remark}
\newtheorem{defn}[thm]{Definition}
\newtheorem{qst}[thm]{Question}
\newtheorem*{qst*}{Question}
\title{Transverse $J$-holomorphic curves in nearly Kähler $\mathbb{CP}^3$}
\author{Benjamin Aslan}
\date{}
\newcommand{\Addresses}{%
 \bigskip
 \footnotesize
 \textsc{Department of Mathematics, University College London,
  UK}\par\nopagebreak
 \textit{E-mail address: }\texttt{benjamin.aslan.17@ucl.ac.uk}}

\begin{document}
\maketitle
\begin{abstract}
\noindent
 $J$-holomorphic curves in nearly Kähler $\mathbb{CP}^3$ are related to minimal surfaces in $S^4$ as well as associative submanifolds in $\Lambda^2_-(S^4)$. We introduce the class of transverse $J$-holomorphic curves and establish a Bonnet-type theorem for them. We classify flat tori in $S^4$ and construct moment-type maps from $\mathbb{CP}^3$ to relate them to the theory of $\uo$-invariant minimal surfaces on $S^4$.
\end{abstract}
\section{Introduction}
In symplectic geometry, the theory of $J$-holomorphic curves is a fundamental topic, which has for example led to the construction of Gromov-Witten invariants. From a Riemannian point of view, it is desirable to have examples of minimal surfaces in Einstein manifolds. One class of such examples comes from complex geometry by considering holomorphic curves in Einstein-Kähler manifolds. Nearly Kähler six-manifolds are Einstein with positive Einstein constant and $J$-holomorphic curves provide a class of examples of minimal surfaces inside them. However, they are neither symplectic nor complex and apart from local statements not much is known about $J$-holomorphic curves in general almost complex manifolds. One feature that is shared with the symplectic case is that the nearly Kähler structure equations guarantee that the volume remains constant on connected components of the moduli space of $J$-holomorphic curves \cite{verbitsky2013pseudoholomorphic}. 
Nevertheless, a general, in-depth theory of $J$-holomorphic curves in nearly Kähler manifolds seems out of reach at the moment, which is why most work has been concerned with studying them in the homogeneous examples. Historically, $M=S^6$ has been the example studied the most, set-off by Bryant's construction of torsion-free $J$-holomorphic curves as integrals of a holomorphic differential system on $\widetilde{\mathrm{Gr}}(2,\R^7)$ \cite{bryant1982submanifolds}. More recently, questions about certain components of the moduli space of $J$-holomorphic curves have been tackled \cite{fernandez}.\par
$J$-holomorphic curves in the twistor spaces $\mathbb{CP}^3$ and the flag manifold $\mathbb{F}$ are related to special submanifolds in two different geometric settings. Firstly, it is a general result that the cone of a $J$-holomorphic curve in a nearly Kähler manifold $M$ gives an associative submanifold in the $G_2$-cone $C(M)$. Remarkably, $J$-holomorphic curves in $\cp$ and $\mathbb{F}$ also give rise to to complete associative submanifolds in the Bryant-Salamon spaces $\Lambda^2_-(S^4)$ and $\Lambda^2_-(\mathbb{CP}^2)$ \cite{karigiannis2005calibrated}. Secondly, via the Eells-Salamon correspondence \cite{eells1985twistorial}, $J$-holomorphic curves in the twistor spaces are in correspondence with minimal surfaces in $S^4$ and $\mathbb{CP}^2$. 
Somewhat similar to torsion-free curves in $S^6$, the space $\cp$ admits a particular class of $J$-holomorphic curves called superminimal curves. They can also be constructed via integrals of a holomorphic differential system and admit a Weierstraß parametrisation \cite{bryant1982conformal}. In contrast, have been studied with methods of integrable systems \cite{perus1992minimal} as minimal surfaces in $S^4$. The aim of this article is to describe various results on non-superminimal surfaces in $S^4$ from the twistor perspective, building on F. Xu's work on $J$-holomorphic curves in nearly Kähler $\mathbb{CP}^3$ \cite{xu2010pseudo}.
Adapting the twistor perspective means that dimension of the ambient space is increased while the second order PDE describing a minimal surface in $S^4$ is reduced to a first order PDE for $J$-holomorphic curves in $\cp$. 
Since the key geometric feature we exploit is the twistor fibration $\cp\to S^4$, it is to be expected that similar techniques can be applied to the flag manifold.\newpage
\subsection*{Summary of Results}
By Cartan-Kähler theory, $J$-holomorphic curves in an almost complex manifold of dimension 6 are locally described by four functions of one variable. Two functions ${\alpha_-},{\alpha_+}$ of two variables that satisfy a Laplace-type equation are parametrised by four functions of one variable too. By using an appropriate adaption of frames, we distill such two functions ${\alpha_-}$ and ${\alpha_+}$ geometrically. The twistor fibration $\cp\to S^4$ comes with a natural connection $T\cp=\mathcal{H}\oplus \mathcal{V}$. We will focus on $J$-holomorphic curves which are going to be called transverse. In particular, they have the property that they are nowhere tangent to horizontal bundle $\mathcal{H}$ or vertical bundle $\mathcal{V}$, see \cref{defntransverse}.
\\
\par
The material covered in \textbf{section \ref{section background}} provides background and context for the subsequent sections. \textbf{Section \ref{section twistor}} establishes relationships between geometric quantities associated to a minimal surface and those of its twistor lifts, featuring a twistor interpretation of Xu's correspondence of certain $J$-holomorphic curves in $\cp$.
\\
\par
In \textbf{section \ref{section adapting frames}}, we define the two functions ${\alpha_-},{\alpha_+}\colon X\to \R$ for a transverse $J$-holomorphic curve which measure angles between different components of $\dd \Phi$ in $T\cp$. They always satisfy the 2D periodic Toda lattice equation for $\mathfrak{sp}(2)$ which are equivalent to the system
\begin{align*}
  \Delta_0 \mathrm{log}({\alpha^2_-})&=-4(3{\alpha^2_-}+{\alpha^2_+}-2)\gamma^2 \\
  \Delta_0 \mathrm{log}({\alpha^2_+}) &=-4(3{\alpha^2_+}+{\alpha^2_-}-2)\gamma^2 
\end{align*}
where $\gamma^2=({\alpha_-}{\alpha_+})^{-1/2}$ is the conformal factor of the induced metric on $X$ and $\Delta_0$ the Laplacian for the corresponding flat metric.
Viewing the surface in the ambient space $\cp$ means we can characterise additional data equipped to the curve, such as the first and second fundamental form in $\cp$. 
It turns out that both of them can be expressed through the functions ${\alpha_-}$ and ${\alpha_+}$. As an application, we show in \cref{classflatcurves} that any flat $J$-holomorphic torus is a Clifford torus. The second fundamental form $\sff$ is a complex linear tensor and its differential $\bar{\partial} \sff$ can be computed neatly in terms of ${\alpha_-},{\alpha_+}$ and vanishes exactly in points where the curve is non-transverse. The normal bundle $\nu$ of a $J$-holomorphic curve carries a natural holomorphic structure. We show that the normal bundles of all transverse $J$-holomorphic tori are isomorphic to each other, in particular all of them admit a holomorphic section. \\ 
\par
\textbf{Section \ref{section bonnet}} features a proof of a Bonnet-type theorem for $J$-holomorphic curves stating that the first and second fundamental form on $X$ are determined by ${\alpha_-}$ and ${\alpha_+}$. The essence of \cref{Bonnet} is that if $X$ is simply-connected then a $\C^*$-family of transverse $J$-holomorphic curves can be recovered from a solution to the 2D Toda lattice equation. Roughly speaking, this can be seen as a complex analogue to the statement that curves in $\R^3$ are essentially classified by their curvature and torsion.\\
\par
To get hold of examples of transverse $J$-holomorphic curves we impose an arbitrary $\uo$-symmetry on them in \textbf{section \ref{u1 section}}, given by a certain element $\xi \in \mathfrak{sp}(2)$.
Such an action commutes with a $\mathbb{T}^2$-action of automorphisms on $\cp$.
Given a $\mathbb{T}^3$-action on a torsion-free $G_2$ manifold $M^7$ the multi-moment maps give rise to a local homeomorphism $M^7/{\mathbb{T}^3} \to \R^4$ \cite[Theorem 4.5]{madsen2012multi}. The only known example of a nearly Kähler manifold admitting a $\mathbb{T}^3$-action is $S^3\times S^3$ whose geometry has been described with multi-moment maps by K. Dixon \cite{dixon2019multi}. Assuming a $\mathbb{T}^2$ symmetry on a nearly Kähler manifold $M^6$ is less restrictive but the corresponding multi-moment map $\nu\colon M^6\to \R$, introduced by G. Russo and A. Swann \cite{russo2019nearly}, is only real-valued. The question arises if there is a geometric construction of a map into $\R^4$ which descends to a local homeomorphism to the $\mathbb{T}^2$ quotient of $M$, at least away from a singular set. We construct a map $p\colon \taction {(\cp\setminus \mathcal{S})}\to \R^4$ for a certain singular set $\mathcal{S}$ and in \cref{fibrationlemma} it is shown that $p$ descends to a branched double cover from $ _{{{\mathbb{T}^2}} }\backslash ^{\cp\setminus \mathcal{S}}$ onto its image $D\subset \R^4$. $p$ maps $\uo$-invariant $J$-holomorphic curves in $\cp$ to solutions of the 1D Toda lattice equation for $\mathfrak{sp}(2)$ in $D$. Derived from the Lax representation of this equation one derives two preserved quantities, giving rise to a map $u\colon D\to \R^2$. \par
If one equips $\cp$ with its Kähler structure then a $\mathbb{T}^2$-action gives rise to a symplectic moment map whose image is a quadrilateral. Composing $u$ with $p$ gives rise to a $\mathbb{T}^2$ invariant map $P\colon \mathbb{CP}^3 \to \R^2$ whose fibres contain $\uo$-invariant $J$-holomorphic curves and whose image is a rectangle $\bar{\mathcal{R}}\subset\R^2$. Just as in the symplectic case, the fibre of $P$ degenerate over the boundary $\partial \bar{\mathcal{R}}$ and are geometrically distinguished sets. In fact, \cref{boundary R} relates $P^{-1}(\partial \bar{\mathcal{R}})$ to the nearly Kähler multi-moment map $\nu$, Clifford tori and certain families of minimal tori in $S^4$ discovered by B. Lawson \cite{lawson1970complete}. 
\\
\\
\begin{minipage}{\linewidth}
\textbf{Acknowledgement.}The author is grateful to his PhD supervisors Jason Lotay and Simon Salamon for their advice and support. This work was supported by the Engineering and Physical Sciences Research Council [EP/L015234/1], the EPSRC Centre for Doctoral Training in Geometry and Number Theory (The London School of Geometry and Number Theory), University College London.
\end{minipage}
\section{Background}
\label{section background}
\subsection{Nearly Kähler six-manifolds}
Nearly Kähler manifolds are a special class of almost Hermitian manifolds $(M,g,J)$ satisfying the equation
\[\nabla_{\xi} J (\xi)=0\]
where $\nabla$ is the Levi-Civita connection on $M$ and $\xi$ denotes any tangent vector on $M$.
Nearly Kähler manifolds have many desirable properties from different points of view. Nearly Kähler manifolds which are not Kähler are called strictly nearly Kähler manifolds. Since Kähler and strictly nearly Kähler manifolds are very different, we will assume that $M$ is strictly nearly Kähler. Furthermore, we assume that $n=6$ since this case of fundamental importance for various reasons.
Firstly, they are Einstein with positive Einstein constant and vanishing first Chern class \cite{gray1976structure}. Furthermore, if $M^6$ is nearly Kähler then the Riemannian cone $C(M)$ carries a torsion-free $G_2$ structure.
Another reason to focus on the case $n=6$ is that they are one of the building blocks of higher-dimensional nearly Kähler manifolds, due to a result by P. Nagy \cite{nagy2002nearly}. Every nearly Kähler manifold carries a unique connection $\nk$ with skew-symmetric torsion and holonomy contained in $\mathrm{SU}(3)$, i.e. $\nk g =\nk J =\nk \psi=0$. \par
Examples of nearly Kähler manifolds are scarce. In fact, there are only six known examples of compact nearly Kähler manifolds.
 If $M=G/H$ is a homogeneous strictly nearly Kähler manifold of dimension six, then $M$ is an element of the following list \cite[Theorem 1]{butruille2010homogeneous}
 \begin{itemize}
 \item $G=G_2$ and $H=\mathrm{SU}(3)$ such that $M=S^6$
 \item $G=S^3\times S^3 \times S^3$ and $H=\{(g,g,g)\mid g \in S^3\}$ such that $M=S^3\times S^3$
 \item $G=\mathrm{Sp}(2)$ and $H=S^1\times S^3$ such that $M=\cp$
 \item $G=\mathrm{SU}(3)$ and $H=\mathbb{T}^2$ such that $M=\mathbb{F}$ is the manifold of complete complex flags of $\C^3$.
 \end{itemize}
 In each case, the identity component of the group of nearly Kähler automorphisms is equal to $G$ and only for $S^3\times S^3$ there is a finite subgroup $\Gamma$ of $G$ acting freely on $M$, giving rise to locally homogeneous examples \cite{cortes2015locally}.
 In addition, there are two known examples of compact, simply-connected nearly Kähler manifolds which are not homogeneous but admit a cohomogeneity-one action. They have been constructed by L. Foscolo and M. Haskins via cohomogeneity one actions on $S^3\times S^3$ and $S^6$ \cite{foscolo2015new}. 
\subsection{Twistor spaces}
\label{TwistorSpaces}
We briefly review a few fundamental results on four-dimensional twistor theory, see \cite{eells1985twistorial}.
To each even-dimensional Riemannian manifold $N$ one can associate a twistor space $Z_{\pm}(N)$, which is a fibre bundle over $N$. The fibre of this bundle over $x$ is given by
\begin{align}
\begin{split}
\label{twistordef}
Z_{\pm}(N)=\{J_x\colon T_xN\to T_xN\mid &J_x^2=-1,\quad g(J_x v,J_x w)=g(v,w) \\ &J_x \text{ preserves/reverses orientation}\}.
\end{split}
\end{align}
Unless specified otherwise, by twistor space $Z(N)$ we refer to the convention of orientation reversing endomorphisms, i.e. $Z_-(N)$. 
From now on, let $N$ be four-dimensional such that the twistor space $Z(N)$ can be identified with elements in $\Lambda^2_{-}(N)$ of unit length. It is a sphere bundle over $N$ which inherits the Levi-Civita connection from $\Lambda^2(N)$. Hence $TZ(N)$ splits into a vertical and horizontal subbundle $TZ(N)=\mathcal{H}\oplus \mathcal{V}$. It turns out that $Z(N)$ possesses a natural almost complex structure $J_1$ known as the Atiyah-Hitchin-Singer almost complex structure. It splits into a sum of almost complex structures on $\mathcal{H}$ and $\mathcal{V}$. Reversing $J_1$ on $\mathcal{V}$ defines the Eells-Salamon almost complex structure $J_2$. The two almost complex structures $J_1$ and $J_2$ are fundamentally different.
It turns out that the twistor space of $N=S^4$ with the round metric is $\mathbb{C}\mathbb{P}^3$. The fibration $\mathbb{C}\mathbb{P}^3\to S^4$ is constructed by composing the map
\[\mathbb{C}\mathbb{P}^3\to \mathbb{H} \mathbb{P}^1,\quad [z_0:z_1:z_2:z_3]\mapsto [z_0+jz_1:z_2+jz_3]\]
with a diffeomorphism $\mathbb{H}\mathbb{P}^1\cong S^4$. 
The twistor space of $N=\mathbb{C}\mathbb{P}^2$ is the flag manifold $\mathrm{SU}(3)/\mathbb{T}^2$. In both cases, $J_1$ gives rise to the well-known Kähler structures on the spaces while $J_2$ belongs to nearly Kähler structures. The splitting $TZ(N)=\mathcal{H}\oplus \mathcal{V}$ is parallel with respect to $\nk$ when $N=S^4$ or $\mathbb{CP}^2$.
 For any immersion $f\colon X \to N^4$ the differential $\dd f$ defines a lift, called the Gauss lift, $\hat{\varphi}$ from $X$ into the oriented Grassmannian bundle $\widetilde{\mathrm{Gr}}_2(TN)$. This bundle in turn projects to $Z(N)$ such that 
by composition with $\hat{\varphi}$ one obtains a map $\varphi\colon X \to Z(N)$ which is called the twistor lift of $f$.
\[
\begin{tikzcd}
                                    & \widetilde{\mathrm{Gr}}_2(TN) \arrow[ld] \arrow[dd] \\
Z(N) \arrow[rd]                            &                              \\
X \arrow[r, "f"] \arrow[u, "\varphi"] \arrow[ruu, "\hat{\varphi}"] & N                            
\end{tikzcd}
\]
 There is a general relation between the Riemannian geometry of $N$ and (almost) complex geometry of $Z(N)$ which is exemplified by the following result.
\begin{prop}[Eells-Salamon]\upshape{\cite[Corollary 5.4]{eells1985twistorial}}
\label{eells-salamon}
 Let $X$ be a Riemann surface and $N$ be a four-dimensional Riemannian manifold. Then $f\colon X\to N$ is a minimal branched immersion if and only if $\varphi \colon X \to Z_{\pm}(N)$ is a $J_2$-holomorphic non-vertical curve.
 \end{prop}
 Note that if $f\colon X \to N$ is a branched minimal immersion, i.e. minimal immersion off a discrete set of points, then there is a rank two subbundle $E\subset f^*(TN)$ which contains $\mathrm{d}f(TX)$ so the Gauß lift is still well defined in this case. 
 Furthermore, it should be mentioned that, since the domain is two-dimensional, a branched minimal immersion is the same thing as a conformal harmonic map.
 Via this correspondence, branched minimal surfaces in $S^4$ are identified with $J_2$-holomorphic curves in the nearly Kähler twistor spaces $\cp$. 
 \subsection{The nearly Kähler structure on $\cp$}
\label{Xusubsection}
As indicated in the previous section, the nearly Kähler structure on $\cp$ can be defined via twistor structure. For explicit computations it is convenient to define the nearly Kähler from the homogeneous space structure $\cp=\mathrm{Sp}(2)/{S^1\times S^3}$.
Identify $\mathbb{H}^2$ with $\C^4$ via $\mathbb{H}=\C\oplus j \C$. This identification gives an action of $\mathrm{Sp}(2)$ on $\C^4$ which descends to $\cp$ and acts transitively on that space. The stabiliser of the element $(1,0,0,0)$ is 
\[\left\{ \begin{pmatrix} z & 0 \\ 0 & q \end{pmatrix} \mid z \in S^1\subset \C, \quad q \in S^3\subset \mathbb{H}\right\}\]
which shows $\cp=\mathrm{Sp}(2)/{S^1\times S^3}$. 
More specifically, consider the Maurer-Cartan form on $\mathrm{Sp}(2)$ which can be written in components as
\begin{align}
\label{MCeq}
\Omega_{MC}=\begin{pmatrix} i\rho_1 +j\overline{\omega_3}& -\frac{\overline{\omega_1}}{\sqrt{2}}+j\frac{\omega_2}{\sqrt{2}}\\ \frac{\omega_1}{\sqrt{2}}+j\frac{\omega_2}{\sqrt{2}} & i\rho_2+j\tau \end{pmatrix}. 
\end{align}
Here, $\omega_1,\omega_2,\omega_3$ and $\tau$ are complex-valued and $\rho_1$ as well as $\rho_2$ are real-valued one-forms on $\mathrm{Sp}(2)$. 
The equation 
\[\dd \Omega_{MC}+\frac{1}{2}[\Omega_{MC},\Omega_{MC}]=0 \]
implies the following differential identities for the components of $\Omega_{MC}$
\begin{align}
\label{firststructure}
 \mathrm{d}\begin{pmatrix} \omega_1 \\ \omega_2 \\ \omega_3 \end{pmatrix}=\underbrace{-\begin{pmatrix} i(\rho_2-\rho_1) &-\overline{\tau} & 0 \\ \tau & -i(\rho_1+\rho_2) & 0 \\ 0 & 0 & 2i\rho_1 \end{pmatrix}}_{A_\omega:=} \wedge\begin{pmatrix} \omega_1 \\ \omega_2 \\ \omega_3 \end{pmatrix} +\begin{pmatrix} \overline{\omega_2\wedge \omega_3} \\ \overline{\omega_3\wedge \omega_1} \\ \overline{\omega_1\wedge \omega_2} \end{pmatrix}.                                                                                                                                                                                                                
\end{align}
Let furthermore
\[
\begin{pmatrix} \kappa_{11} & \kappa_{12} \\ \kappa_{21} & \kappa_{22} \end{pmatrix}=\begin{pmatrix} i(\rho_2-\rho_1) & -\overline{\tau} \\ \tau & -i(\rho_1+\rho_2) \end{pmatrix},
\]
so one obtains
\[\mathrm{d} \begin{pmatrix} \kappa_{11} & \kappa_{12} \\ \kappa_{21} & \kappa_{22} \end{pmatrix}=-\begin{pmatrix}
                                                  \kappa_{11} & \kappa_{12} \\ \kappa_{21} & \kappa_{22}
                                                  \end{pmatrix} \wedge
                                                  \begin{pmatrix}
                                                  \kappa_{11} & \kappa_{12} \\ \kappa_{21} & \kappa_{22}
                                                  \end{pmatrix}
+\begin{pmatrix} \omega_1\wedge \overline{\omega_1}-\omega_3\wedge\overline{\omega_3} & \omega_1\wedge \overline{\omega_2} \\
                                                 \omega_2\wedge \overline{\omega_1} & \omega_2\wedge \overline{\omega_2}-\omega_3\wedge \overline{\omega_3} \end{pmatrix}.\]
Let finally $\kappa_{33}=2i\rho_1$ which satisfies $\dd \kappa_{33}=\omega_1\wedge \overline{\omega_1}+\omega_2\wedge \overline{\omega_2}-2\omega_3\wedge \overline{\omega_3}$.
The nearly Kähler structure on $\cp$ is defined by declaring the forms $s^*(\omega_1),s^*(\omega_2)$ and $s^*(\omega_3)$ to be unitary $(1,0)$ forms for any local section $s$ of the bundle $\spp\to \cp$. The resulting almost complex structure and metric do not depend on the choice of $s$. The matrix $A_{\omega}$ is in fact the connection matrix of the nearly Kähler connection $\nk$ in the frame dual to the local co-frame $(\omega_1,\omega_2,\omega_3)$.
Note that the forms $(\omega_1,\omega_2,\omega_3,\tau)$ are complex-valued invariant forms on $\mathrm{Sp}(2)$ and as such they can be seen as elements in $\mathfrak{sp}(2)^\vee \otimes \C$. They span different root spaces with respect to the maximal torus $S^1\times S^1$, see \cref{figure1sp2}. 
\section{Properties of twistor lifts}
\label{section twistor}
As seen in \cref{twistordef}, there is a notion of a negative and positive twistor space for a Riemannian four-manifold $N^4$. In general, they can be different fibre bundles over $N^4$ and \cref{twistor second ff} expresses how the second fundamental form of a surface in $N^4$ is related to the lift into both twistor spaces. However, for $N=S^4$ both twistor space are isomorphic as fibre bundles which we exploit to give a twistor interpretation of Xu's correspondence.
\subsection{Encoding the second fundamental form}
Let $N^4$ be a Riemannian manifold and $f \colon X^2\to N^4$ be an isometric immersion.
The second fundamental form $\sff$ encodes local geometric information about $f$. The aim of this section is to relate the norm of $\sff$ to quantities defined for the twistor lift of $f$.
Let $\nu$ be the normal bundle of $TX$ in $f^*(TN)$. Locally, fix an oriented orthonormal frame $\{e_1,e_2,e_3,e_4\}$ such that $\{e_1,e_2\}$ is an oriented basis of $f^*(TX)$. 
Denote by $\omega_{ij}(v)=g(\nabla_v e_i,e_j)$ the locally defined connection one-forms of the Levi-Civita connection $\nabla$ on $N$. Furthermore, let $\omega_{ijk}=\omega_{ij}(e_k)$. Then $\omega_{ijk}=-\omega_{jik}$ and if $i,k\in \{1,2\}$ and $j\in \{3,4\}$ then $\omega_{ijk}=\omega_{kji}$ holds by torsion-freeness.
The vertical component of the twistor lift can be computed in the following way \cite{friedrich1984surfaces}
\begin{align}
\label{verliftminus}
(\mathrm{d}\varphi_-)^\ver=\frac{\omega_{13}+\omega_{24}}{2}y_5+\frac{\omega_{14}-\omega_{23}}{2}y_6.
\end{align}
Here $\{y_5,y_6\}$ is a local orthonormal basis for the vertical space of $Z$, induced from the choice of $\{e_1,\dots,e_4\}$. One can regard $(\mathrm{d}\varphi)^\ver$ as a one-form with values in $\mathrm{Hom}_{\C}(TX,\nu)$. Under this identification,
\[y_5(e_1,e_2,e_3,e_4)=(e_3,e_4,-e_1,-e_2), \quad y_6(e_1,e_2,e_3,e_4)=(e_4,-e_3,-e_2,-e_1).\]
The almost complex structures $J_1$ and $J_2$ act by $J_1(y_5,y_6)=(-y_6,y_5)$ and $J_2(y_5,y_6)=(y_6,-y_5)$ on the vertical bundle.
Let $f\colon X\to N$ be an immersion. Changing between $\varphi_-$ and $\varphi_+$ amounts to changing the orientation of $N$, i.e. swapping the indices $3\leftrightarrow 4$. This gives the analogous formula 
\begin{align}
\label{verliftplus}
(\mathrm{d}\varphi_+)^\ver=\frac{\omega_{14}+\omega_{23}}{2}y_5+\frac{\omega_{13}-\omega_{24}}{2}y_6.
\end{align}
\begin{prop}
\label{twistor second ff}
\[|\dd \varphi_-^{\mathcal{V}}|^2+|\dd \varphi_+^{\mathcal{V}}|^2=\frac{1}{2}|\sff|^2.\] Furthermore, if $f$ is minimal then 
 $|\dd \varphi_-^{\mathcal{V}}|=|\dd \varphi_+^{\mathcal{V}}|$ if and only if $\sff $ takes values in a real line bundle contained in $\nu$. 
\end{prop}
\begin{proof}
As a consequence of \cref{verliftminus} and \cref{verliftplus} we have
 \begin{align*}
 4|\dd \varphi_-^{\mathcal{V}}|^2&=(\omega_{131}+\omega_{241})^2+(\omega_{132}+\omega_{242})^2+(\omega_{141}-\omega_{231})^2+(\omega_{142}-\omega_{323})^2 \\
 4|\dd \varphi_+^{\mathcal{V}}|^2&=(\omega_{141}+\omega_{231})^2+(\omega_{142}+\omega_{232})^2+(\omega_{131}-\omega_{241})^2+(\omega_{132}-\omega_{343})^2
 \end{align*}
which implies the first statement.
Hence, $|\dd \varphi_-^{\mathcal{V}}|=|\dd \varphi_-^{\mathcal{V}}|$ if and only if 
\[\omega_{142}(\omega_{131}-\omega_{232})+\omega_{132}((\omega_{242}-\omega_{141})=0.\]
This condition is satisfied if and only if $\sff_3$ and $\sff_4$ commute.
Here, 
\[\sff_j=(\omega_{ijk})_{i,k=1,2}, \quad j\in\{3,4\}.\]
This in turn is equivalent to $\sff_3$ and $\sff_4$ being simultaneously diagonalisable. So, assume that $\{e_1,e_2\}$ is a positive local frame in which $\sff_3$ and $\sff_4$ are diagonal, i.e. $\sff_3=\mathrm{diag}(\omega_{131},-\omega_{131})$ and $\sff_4=\mathrm{diag}(\omega_{141},-\omega_{141})$ due to minimality. Hence $\sff$ takes values in the real line bundle which is contained in $\nu$ and locally spanned by 
$\omega_{131} \otimes e_3 + \omega_{141} \otimes e_4.$
\end{proof}
\begin{crl}
 Let $x\in X$, then $f$ is totally geodesic in $x$ if and only if both twistor lifts are horizontal.
\end{crl}
\begin{crl}
\label{twistor lifts totally geodesic}
 If $f$ takes values in a totally geodesic submanifold then $|\dd \varphi_-^{\mathcal{V}}|=|\dd \varphi_+^{\mathcal{V}}|$.
\end{crl}
\begin{lma}
\label{G-inv}
 Let $X\subset Y\subset Z$ all be compact with $Y$ totally geodesic and let $G$ act by isometries $X$ and $Z$. Then $Y$ is also $G$-invariant.
\end{lma}
\begin{proof}
 Let $\nu$ be the normal bundle of $X$ in $Y$. Since $X,Y$ are compact, $Y$ is the image of $\exp\colon \nu \to Z$. $\exp$ is $G$-invariant since $G$ acts by isometries, which implies the statement.
\end{proof}
 \subsection{Twistor interpretation of Xu's correspondence}
We give a brief overview of Xu's work on $J$-holomorphic curves in $\cp$ and provide a twistor interpretation of his correspondence of superminimal curves and curves with vanishing torsion \cite{xu2010pseudo}.
A superminimal curve is a $J$-holomorphic curve $\varphi\colon X\to \cp$ which is tangent to $\mathcal{H}$, i.e. $\dd \varphi\colon$ maps $TX$ to $\varphi^*\mathcal{H}$. These curves are special since are holomorphic for $J_1$ too and R. Bryant found a Weierstraß parametrisation for them, i.e. each of them is a projective line or parametrised by
\begin{align}
\label{thmbryant}
\Theta(f,g)=[1,f-\frac{1}{2}g\Bigl(\frac{\mathrm{df}}{\mathrm{dg}}\Bigr),g,\frac{1}{2}\Bigl(\frac{\mathrm{d}f}{\mathrm{d}g}\Bigr)].
\end{align}
for $f,g$ meromorphic functions on $X$ with $g$ being non-constant \cite[Theorem F]{bryant1982conformal}.
For a general $J$-holomorphic curve $\varphi$ there is a unique holomorphic line bundle $L\subset \varphi^*\mathcal{H}$ which contains the projection of $\dd{\varphi}(TX)$ to $\varphi^*\mathcal{H}$. Let $N$ be the quotient bundle $\varphi^*\mathcal{H}/L$, which naturally carries a holomorphic structure. Define 
\[P=\{p \in \varphi^* \spp \mid \omega_1| _p=0\}\]
which is an $S^1\times S^1$ sub-bundle of $\varphi^* \spp$. The bundle $P$ can be equipped with a connection such that the forms $\tau$ and $\omega_2$ restrict to basic forms on $P$. By multiplying $\bar{\tau}$ by an appropriate section in $L^\vee \otimes N$ one obtains the form $\Pi \in \Omega^{1,0}(X,L^\vee\otimes N)$ which is shown to be holomorphic \cite[Lemma 3.2.3]{xu20063}.
\begin{prop}
\label{xucorrespondence}
 In $\cp$, there is a one-to-one correspondence between horizontal holomorphic curves and $J_2$-holomorphic curves on which $\Pi$ vanishes. 
\end{prop}
%
The correspondence of \cref{xucorrespondence} can be entirely described in terms of twistor lifts.
All statements from \cref{TwistorSpaces} apply to the negative as well as to the positive twistor space.
The twistor theory of $S^4$ has the particularity that both the negative and the positive twistor space can be identified with $\mathbb{CP}^3$. To distinguish the two spaces as bundles over $S^4$ we denote them by $\mathbb{CP}^3_{\pm}$. Both spaces are quotients of $\mathrm{Sp}(2)$ by different but conjugate subgroups. 
\[\begin{tikzcd}
                         & Sp(2)/{S^1\times S^1} \arrow[ld] \arrow[rd] &                         \\
\mathbb{CP}^3_{-}=\mathrm{Sp}(2)/{S^1\times S^3} &                       & \mathbb{CP}^3_{+}=\mathrm{Sp}(2)/{S^3\times S^1}
\end{tikzcd}.\]
For a $J_2$-holomorphic curve $\varphi_-\colon X\to \mathbb{CP}^3_-$ Xu constructs a lift $\tilde{\varphi}_-\colon X \to Sp(2)/{S^1\times S^1}$ and then considers the projection onto $\mathbb{CP}^3_+$ which yields a map $X \to \mathbb{CP}^3_+$. He constructs a similar lift $\tilde{\varphi}_+$ when starting with a curve in $\mathbb{CP}^3_+$ and shows that both constructions are inverse to each other.
Observe that $Sp(2)/{S^1\times S^1}$ is nothing but $\widetilde{\mathrm{Gr}}_2(S^4)$. Let $\pi_{\pm}$ be the projection of $\mathbb{CP}^3_{\pm}$ onto $S^4$. It turns out that the lifts $\varphi_{\pm}$ equal the Gauß lift of $\pi_{\pm}\circ \varphi$ into $\mathrm{Sp}(2)/{S^1\times S^1}$. This immediately shows that the two constructions are inverse to each other since the transformation leaves the underlying map into $S^4$ unchanged. In fact, this procedure gives a way to pass between $J_2$-holomorphic curves in $\mathbb{CP}^3_+$ and $\mathbb{CP}^3_-$ due to \cref{eells-salamon}. However, when starting with a horizontal $J_2$-holomorphic curve in $\mathbb{CP}^3_-$ the resulting curve in $\mathbb{CP}^3_+$ need not be horizontal. In fact the tangent bundle of $F$ admits a natural splitting 
\begin{align}
\label{Fconn}
 T(Sp(2)/{S^1\times S^1})=\textbf{H}\oplus \textbf{V}_+ \oplus \textbf{V}_-
\end{align}
 which is a connection of the fibration $Sp(2)/{S^1\times S^1}\to S^4$ such that
\begin{align*}
 \mathbf{V}_{\pm}=\mathrm{ker}(\mathrm{d}\pi_{\mp}), \qquad \pi_{\pm}^*(\mathcal{H}_{\pm})=\textbf{H}, \qquad \pi_{\pm}^*(\mathcal{V}_{\pm})=\mathbf{V}_{\pm}.
\end{align*}
If $\{f_1,f_2,f_3,f_4\}$ denotes a frame dual to $\{\omega_1,\omega_2,\omega_3,\tau\}$ then $\textbf{H}$ is locally spanned by $f_1$ and $f_2$, $V_-$ by $f_3$ and $V_+$ by $f_4$.
Consider a $J_2$-holomorphic curve $\varphi_-\colon X \to \mathbb{CP}^3_-$ and the Gauß lift $\hat{\varphi}_-\colon X \to \spp/{S^1\times S^1}$. The curve is horizontal if the lift does not have a component in $V_-$. \cref{Fconn} describes the splitting of $\mathfrak{sp}(2)/{\mathfrak{t}^2}$ into root spaces. In fact,
\cref{MCeq} reveals that the component in $V_+$ vanishes if and only if \[\varphi^*\tau=0 \quad \Leftrightarrow \quad \Pi=0\] 
on $X$ which proves \cref{xucorrespondence}.
Finally, observe that \cref{firststructure} implies that $\Pi$ is equal to the second fundamental form of $L$ in $\varphi^*\mathcal{H}$ which is an element in $\Omega^{1,0}(X,\mathrm{Hom}(L,N))$ because $L$ is a holomorphic sub bundle of $\varphi^*\mathcal{H}$. The holomorphicity of $\Pi$ can then be related to properties of the curvature tensor of $\varphi^*\mathcal{H}$. 
\section{Adapting frames on $J$-holomorphic curves}
\label{section adapting frames}
\todo{give a twistor interpretation of the symmetry ${\alpha_-}$<->${\alpha_+}$}
Recall from \cref{Xusubsection}
that there are three distinguished classes of $J$-holomorphic curves in $\cp$. The first are curves which are always tangent to the vertical bundle $\mathcal{V}$. They are twistor lines and are parametrised by elements in $S^4$. Since the horizontal bundle $\mathcal{H}$ is of complex rank two it is less restrictive to require a curve being tangent to $\mathcal{H}$. These curves are called superminimal and classified in \cref{thmbryant}. There is a third class of curves, namely those on which $\Pi$ vanishes identically. Such curves are in one to one correspondence with superminimal curves by \cref{xucorrespondence}. Since all of these classes are relatively well understood we are interested in studying $J$-holomorphic curves which do not belong to these three classes. A general $J$-holomorphic curve will have isolated points in which it is either horizontal, vertical or satisfies $\Pi=0$. These points will appear as singularities in our treatment and the statements we will show are for $J$-holomorphic curves without such singularities which we will call transverse. 
\begin{defn}
\label{defntransverse}
 A $J$-holomorphic curve $\varphi\colon X\to \cp$ is called transverse if one of the equivalent conditions is satisfied
 \begin{itemize}
 \item $\Pi\neq 0$ everywhere and $\varphi$ is nowhere tangent to the horizontal bundle $\mathcal{H}$ or the vertical bundle $\mathcal{V}$
 \item Both twistor lifts of $\pi_-\circ \varphi\colon X\to S^4$ are nowhere horizontal or vertical
 \item The Gauss lift of $\pi_-\circ \varphi$ into $\mathrm{Sp}(2)/{S^1\times S^1}$is nowhere tangent to either bundle $\mathbf{H}$,$\mathbf{V}_-$ or $\mathbf{V}_+$.
 \end{itemize}
\end{defn}
If one is willing to compromise on completeness one can consider the open set away from the singularities.
If $X$ is compact with genus $g$ then the curve is automatically superminimal or satisfies $\Pi=0$ if $g=0$, has non-transverse points if $g\geq 2$ and is either superminimal, satisfies $\Pi=0$ or transverse if $g=1$ \cite[Remark 4.11.]{xu20063}. For this reason, we are mainly interested in the case when $X$ has genus one.\par
 In the case of the nearly Kähler $S^3\times S^3$, the almost product structure gives a natural choice of an $\mathrm{SU}(3)$ frame along any $J$-holomorphic curve. This allows to study the structure equations in terms of a reduced set of functions, classify and identify sub-classes of $J$-holomorphic curves in $S^3\times S^3$ \cite{tohoku2015}. The aim of this section is to provide a similar analysis for $J$-holomorphic curves in $\cp$. Instead of an almost product structure, the parallel splitting $T\cp=\mathcal{H}\oplus \mathcal{V}$ plays a key role in this case.
 By an appropriate frame adaption we will describe transverse $J$-holomorphic curves by two functions ${\alpha_-},{\alpha_+}\colon X\to \R$ which carry local geometric information such as the curvature and second fundamental form of the curve. 
\subsection{The action of $H$ on $\mathfrak{sp}(2)$}
Consider the embedding $i\colon \mathrm{U}(2)\to \mathrm{SU}(3)$ via $A\mapsto \begin{pmatrix} A & 0 \\ 0 & \mathrm{det}(A^{-1}) \end{pmatrix}$. Let $(v_1,v_2,v_3)^T\in \mathbb{C}^3$ with $|v_1|^2+|v_2|^2 \neq 0$ and $|v_3|^2 \neq 0$. Then there is $A\in \mathrm{U}(2)$ such that if $(w_1,w_2,w_3)^T=A (v_1,v_2,v_3)^T$ then $w_2=0$ and $w_3/w_1\in \R^{>0}$. The choice of such an $A$ is unique up to multiplication by an element in the subgroup $K'=\{\mathrm{diag}(e^{i\vartheta},e^{-2i\vartheta},e^{i\vartheta})\}\subset \mathrm{SU}(3)$.
Define the double cover $u\colon H=\mathrm{U}(1)\times \mathrm{Sp}(1) \to \mathrm{U}(2)$ where $u(\lambda,q)$ acts on $\C^2=\C\oplus j \C=\HH$ by $h\mapsto q h \lambda^{-1}$. Let $\rho$ be the action of $\mathrm{U}(1)\times \mathrm{Sp}(1)$ on $V_1=\C^3$ coming from $i\circ u$. Consider the adjoint action of $H\subset \spp$ on $\mathfrak{sp}(2)$. It splits as
\[\mathfrak{sp}(2)=\mathfrak{h}\oplus V_1.\]
The action of $H$ on $\mathfrak{h}$ is the adjoint action while $H$ acts on $V_1$ by $\rho$. Here $V_1$ embeds into $\mathfrak{sp}(2)$ as follows
\[(z_1,z_2,z_3)\mapsto \begin{pmatrix} j \bar{z_3} & -\overline{z_1}+j\overline{z_2} \\ z_1+jz_2 & 0
            \end{pmatrix}.\]
Note that $K=\rho^{-1}(K')=\{\mathrm{diag}(e^{i\theta},e^{i 3 \theta})\}$ and
define $W=\{(v_1,0,v_3)\in V_1 \mid v_3/{v_1}>0\}$. We have proven the following. 
\begin{lma}
\label{linear algebra 1}
For any $\zeta=\eta+(v_1,v_2,v_3)\in \mathfrak{h} \oplus \mathbb{C}^3$ with $(v_1,v_2)\neq (0,0)$ and $v_3 \neq 0$ there is $h\in H$ such that $h\zeta h^{-1}$ lies in $W$. Such an $h$ is unique up to a an element in $K$.
\end{lma}
The adjoint action of $K$ on $\mathfrak{h}$ splits into one-dimensional subspaces. Let
\begin{align*}
V_2&= \left\{\begin{pmatrix} 0 & 0\\ 0 & j w \end{pmatrix}\mid w\in \C \right\} \\
V_3&= \left\{\begin{pmatrix} i x_1 & 0\\ 0 & i x_2 \end{pmatrix}\mid x_1,x_2 \in \R \right\}.
\end{align*}
The action of $K$ on $V_3$ is trivial while it acts on $V_2$ by multiplication of $e^{-6i \theta}$. On the other hand, $K$ acts on $(z_1,0,z_3)\in W$ by multiplication of $e^{2i\theta}$ in each component. Define
\[ \mathfrak{r}=\{ \begin{pmatrix}
          ix_1+j \bar{z_3} & -\bar{z_1} \\
          z_1 & ix_2+j w
          \end{pmatrix} \mid z_1\neq 0,\quad z_3/z_1 \in \R^{>0},\quad w/z_1 \in \R^{>0}
\}.\]
We conclude
\begin{lma}
\label{linear algebra 2}
 Let $v=(z_1,z_2,z_3,w,x_1,x_2) \in \mathfrak{sp}(2)$ with $(z_1,z_2)\neq (0,0), w\neq 0$ and $z_3\neq 0$. Then there is always an element $A\in K$ such that $Av\in \mathfrak{r}$. The choice of such an element $A$ is unique up to multiplication of an element of the subgroup $K_F:=\{\mathrm{diag}(e^{i\theta},e^{i 3 \theta})\mid e^{8i\theta}=1\}\cong \Z_8.$ 
\end{lma}
\subsection{Transverse $J$-holomorphic curves in $\cp$}
Consider a general $J$-holomorphic curve $\varphi\colon X\to \cp$. Denote by $g$ the nearly Kähler metric on $\cp$ which splits as $g_{\mathcal{H}}+g_{\mathcal{V}}$ since the splitting $T\cp=\mathcal{H}\oplus \mathcal{V}$ is orthogonal. We will consider the pull-back metrics of $g,g_{\mathcal{H}},g_{\mathcal{V}}$ to $X$ via $\varphi$ and denote them with the same symbols. Note that $\mathrm{Sp}(2)$ pulls back to an $S^1\times S^3$ bundle over $X$. The structure equations are formulated with respect to differential forms on $\mathrm{Sp}(2)$. To simplify matters, one reduces the $S^1\times S^3$ bundle $\varphi^*(\mathrm{Sp}(2))$ which ensures that additional relations of the differential forms are satisfied.
To begin with, one can reduce the bundle $\varphi^*(\mathrm{Sp}(2))$ to an $S^1\times S^1$ bundle $P$ by imposing the equation $\omega_2=0$, see \cite{xu2010pseudo}. This gives a lift of $\varphi$ into $\mathrm{Sp}(2)/{S^1\times S^1}$. On this reduction, $\tau$ becomes a basic form of type $(1,0)$. Since $\varphi$ is $J$-holomorphic $\dd \varphi(T^{1,0})$ takes values in the subbundle corresponding to the root spaces $\{(2,0),(0,-2),(-1,1)\}$ which are associated to $(\omega_3,\tau,\omega_1)$ under the isomorphism $\Omega^1(\mathrm{Sp}(2),\C)^{\mathrm{Sp}(2)}\cong \mathfrak{sp}(2)^\vee \otimes \C$, see \cref{figure2sp2}. 
Given a semi-simple Lie group with maximal torus $T$ the flag manifold $G/T$ carries a natural $m$-symmetric structure $\tau$ where $m$ is the height of the Lie algebra of $G$. $\tau$ gives rise to a splitting $T_{\C}(G/T)=\oplus [\mathcal{M}_k]$ and a map $\psi \colon X \to G/T$ is called $\tau$-primitve if it satisfies a compatibility condition between the complex structure on $X$ and $\tau$, see \cite{boltonwoodward}.
In the case of $G=\mathrm{Sp}(2)$, $\{(2,0),(0,-2),(-1,1)\}$ is a basis of $[\mathcal{M}_1]$.
\begin{figure}
\centering
\begin{minipage}{0.4 \textwidth}
 \includegraphics[width=5cm]{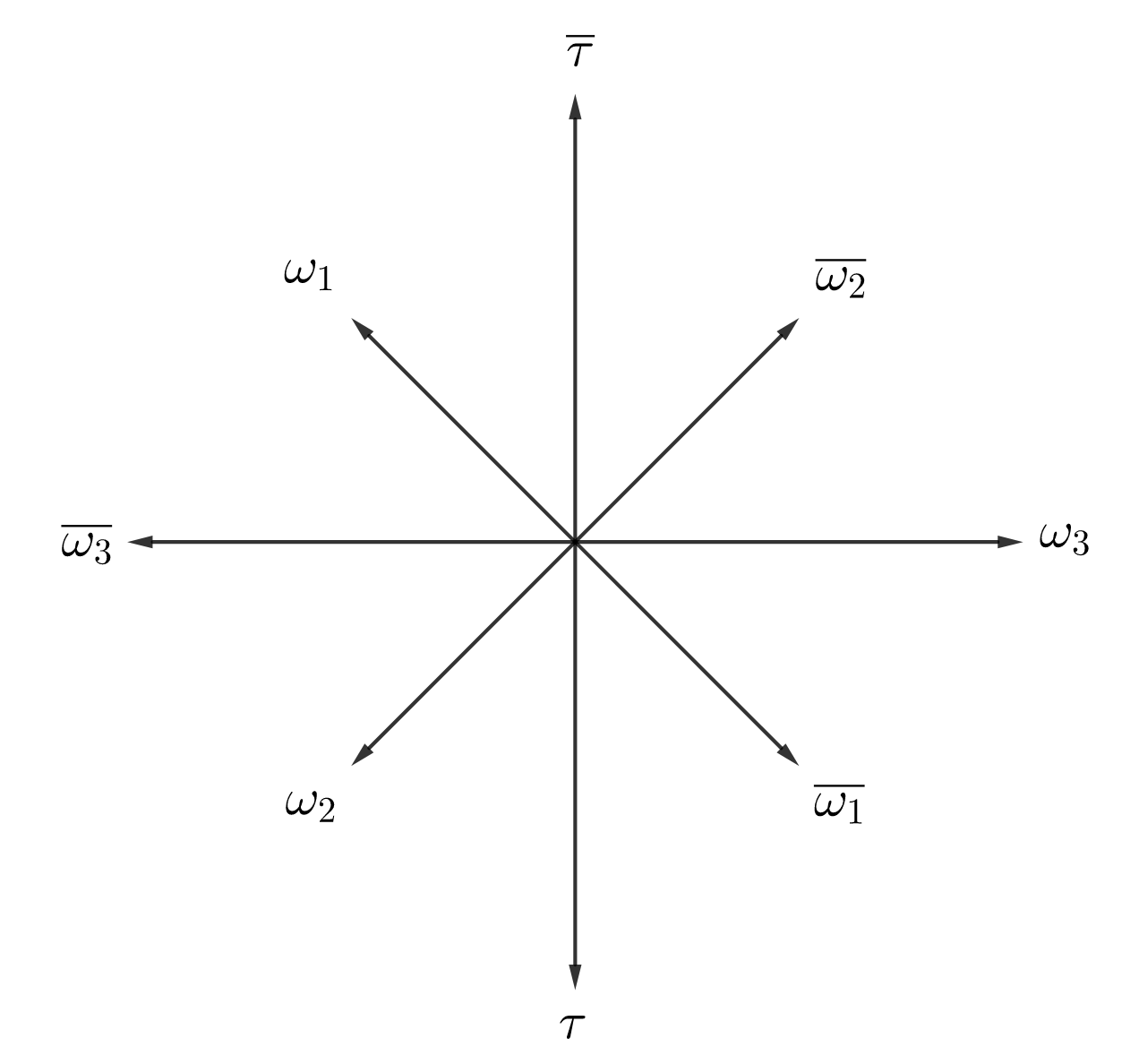}
\caption{Root spaces associated to different components of $\Omega_{MC}$ under the identification $\Omega^1(\mathrm{Sp}(2),\C)^{\mathrm{Sp}(2)}\cong \mathfrak{sp}(2)^\vee \otimes \C$ }
\label{figure1sp2}
\end{minipage}
\hspace{2cm}
\begin{minipage}{0.4 \textwidth}
\includegraphics[width=5cm]{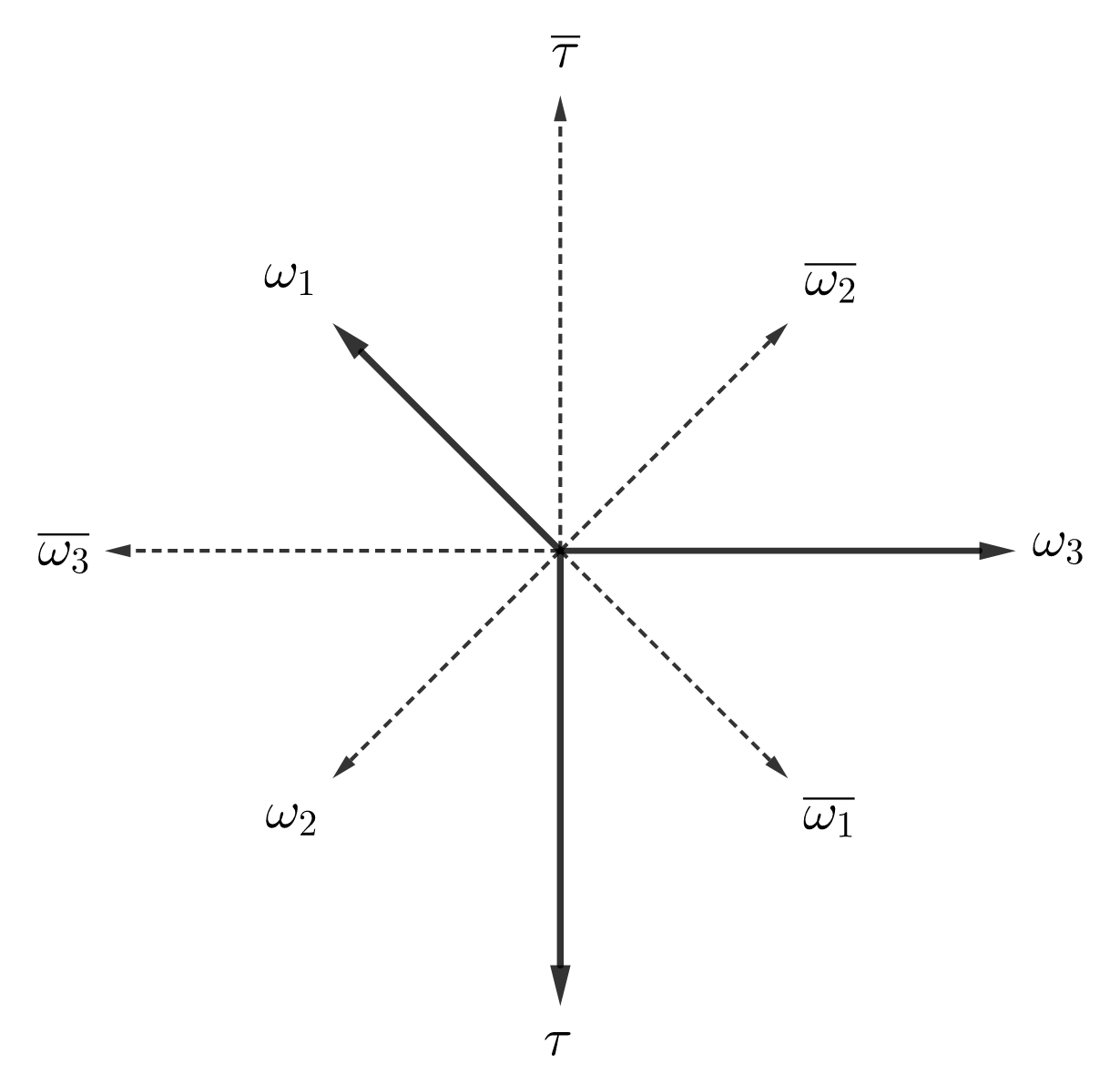}
\caption{The thickened arrows represent a basis of $[\mathcal{M}_1]$ highlighting that transverse $J$-holomorphic curves are in one-to-one correspondence with $\tau$-primitive maps in $\mathrm{Sp}(2)/{S^1\times S^1}$.}
\label{figure2sp2}
\end{minipage}
\end{figure}
\begin{prop}
\label{tauprimitive $J$-holomorphic}
 Any $J$-holomorphic curve $\varphi\colon X \to \cp$ admits a lift into $\mathrm{Sp}(2)/{S^1\times S^1}$ which is $\tau$-primitive in the sense of \cite{boltonwoodward}. Conversely, composing any $\tau$-primitive map $X\to \mathrm{Sp}(2)/{S^1\times S^1}$ with the projection $\mathrm{Sp}(2)/{S^1\times S^1}\to \cp$ gives a $J$-holomorphic curve $X\to \cp$.
\end{prop}
Note that under this identification, the formula \cite[Remark 4.11.]{xu20063} becomes an application of the more general Pluecker formula for $\tau$-primitive maps \cite{boltonwoodward2}.\par
We will now explain how $\varphi^*(\mathrm{Sp}(2))$ can in fact be reduced to a discrete bundle for a transverse $J$-holomorphic curve $\varphi\colon X \to \cp$. We can define the function ${\alpha_-}\colon X \to \mathbb{R}^{>0}$ by 
\[{\alpha_-}(x)=\frac{\|\xi\|_{\mathcal{V}}}{\|\xi\|_{\mathcal{H}}},\]
where $\xi$ is any non-zero vector in $T_xX$. Here $|\cdot|_{\mathcal{H}}$ and $|\cdot|_{\mathcal{V}}$ denote the norms with respect to the metric $g_{\mathcal{H}}$ and $g_{\mathcal{V}}$ on $X$. The value of ${\alpha_-}$ does not depend on this choice because $\varphi$ is $J$-holomorphic. The function ${\alpha_-}$ is a measure for the angle in which $TX$ lies between $\mathcal{H}_X$ and $\mathcal{V}_X$.
In accordance with \cref{linear algebra 1} we now adapt frames in the following way. Over $X$, define the principal bundle $Q$ by the relations
\begin{align*}
 \omega_2&=0 \\
 \omega_3&={\alpha_-} \omega_1.
\end{align*}
The bundle $Q$ has structure group $K\cong S^1$. From now on, we will consider the restrictions of all differential forms to $Q$ without changing the notation. Note that, since the structure group of $Q$ is $K$, the forms $\tau,3\rho_1-\rho_2$ and $\omega_1\wedge{\overline{\omega_1}}$ become basic, i.e. they descend to forms on $\cp$.
The structure equations then yield
\begin{align}
\label{dero1}
 \dd \omega_1 &=-i(\rho_2-\rho_1) \wedge \omega_1 \\
 \dd \omega_2 &=\tau \wedge \omega_1=0 \\
 \dd \omega_3 &=-2i \rho_1 \wedge \omega_3.
\end{align}
Combining all of the equations gives that $\tau$ and $-\dd \mathrm{log}({\alpha_-})+i(-3\rho_1+\rho_2)$ are $(1,0)$-forms since their wedge product with $\omega_1$ vanishes.
Since $\rho_1$ and $\rho_2$ and $-\dd \mathrm{log}({\alpha_-})$ are real-valued we get that 
\[\dd \mathrm{log}({\alpha_-})(J\xi)=(-3\rho_1+\rho_2)(\xi)\]
for any tangent vector $\xi$. In other words,
\begin{align}
\label{dC alpha}
\dd^C \mathrm{log}({\alpha_-})=-3\rho_1+\rho_2.
\end{align}
Note that a-priori this is an equation on $Q$ but it also holds on $X$ since $-3\rho_1+\rho_2$ is a basic form.
\cref{dC alpha} implies that 
\begin{align}
\label{laplacealpha}
-\Delta \mathrm{log}({\alpha_-})=\dd (\rho_2-3 \rho_1).
\end{align}
Here, $\Delta$ denotes the positive definite Laplace operator $\Omega^0(X)\to \Omega^2(X)$.
\subsection{Further reduction of the structure group}
Since $\varphi$ is assumed to be transverse, $\tau$ is nowhere vanishing. On $Q$, both $\tau$ and $\omega_1$ reduce to forms on $X$ with values in the same line bundle. So we can define
\[{\alpha_+}(x)=\frac{|\omega_1(\xi)|}{|\tau(\xi)|}\]
for any $\xi \in T_xX$ which does not depend on $\xi\in T_x X$ since $\tau$ is a $(1,0)$-form. In fact, by \cref{tauprimitive $J$-holomorphic} transverse $J$-holomorphic curves correspond to $\tau$-primitive maps $\hat \varphi\colon X\to \mathrm{Sp}(2)/{S^1\times S^1}$. As seen in \cref{Fconn} 
\[T(\mathrm{Sp}(2)/S^1\times S^1)=H\oplus V_-\oplus V_+.\]
Note that ${\alpha_-}$ is a measure of the angle of $TX$ between $\hat{\varphi}^*(H)$ and $\hat{\varphi}^*(V_-)$ while ${\alpha_+}$ is a measure of the angle of $TX$ between $\hat{\varphi}^*(H)$ and $\hat{\varphi}^*(V_-)$.
This is why we will refer to ${\alpha_-}$ and ${\alpha_+}$ as angle functions.
By \cref{linear algebra 1}, we can adapt frames further. The bundle $Q$ restricts to a $K_F$ bundle $R$ which is characterised by the equation
$\tau={\alpha_+} \omega_1$.
From the structure equations we now get
\begin{align}
\label{dertau}
 \dd \tau &= 2i \rho_2 \wedge \tau \\
 \label{derr1}
 -2i \dd \rho_1 &= (1-2{\alpha^2_-})\omega_1\wedge\overline{\omega_1}\\
 \label{derr2}
 2i \dd \rho_2 &= 2\overline{\tau}\wedge \tau+ \omega_1\wedge\overline{\omega_1}=(-2{\alpha^2_+}+1)\omega_1\wedge\overline{\omega_1}.
\end{align}
Combining \cref{derr1} and \cref{dertau} one infers that $-\dd \mathrm{log}({\alpha_+})+i(-\rho_1+3\rho_2)$ is a $(1,0)$-form. As before, we get
\[\dd^C \log ({\alpha_+})=-\rho_1+3\rho_2.\]
This implies
\begin{align}
\label{laplacebeta}
\Delta \mathrm{log}({\alpha_+})=\dd (\rho_1-3 \rho_2).
\end{align}
Let us summarise the results so far
\begin{lma}
\label{summaryreducedbundle}
 Let $\varphi\colon X \to \cp$ be a transverse $J$-holomorphic curve. Then the bundle $\varphi^*(\spp)$ restricts to a $K_F$ bundle $R$ on which the following equations hold
 \begin{align}
 \label{R equations}
 \begin{split}
 \omega_3&={\alpha_-} \omega_1 \qquad \omega_2=0 \qquad \tau= {\alpha_+} \omega_1 \\
 \rho_1&=\frac{1}{8}(-3\dd^C \log ({\alpha_-})+\dd^C \log({\alpha_+})) \qquad \rho_2=\frac{1}{8}(-\dd^C \log ({\alpha_-})+3\dd^C \log({\alpha_+})).
 \end{split}
 \end{align}
\end{lma}
Reducing the bundle $\mathrm{Sp}(2)$ over a transverse $J$-holomorphic curve can be summarised in the following table.
\begin{table}[H]
\begin{tabular}{l|l|l|l}
Bundle & Structure Group & Reduction characterised by & Restriction implies \\
\hline
$\varphi^*\mathrm{Sp}(2)$& $H=S^1\times S^3$ & & \\
 $P$ & $S^1\times S^1$ & $\omega_1=0$ & $\tau$ \text{is of type }$(1,0)$ \\
 $Q$ & $K\cong S^1$ & $\omega_3={\alpha_-} \omega_1,\omega_2=0$, and \cref{linear algebra 1} & $\dd ^C {\mathrm{log}({\alpha_-})}=-3\rho_1+\rho_2$ \\
 $R$ & $K_F \cong \Z_8$ & Above, $\tau={\alpha_+} \omega_1$ and \cref{linear algebra 2} & $\dd^C {\mathrm{log}({\alpha_+})}=-\rho_1+3\rho_2$
\end{tabular}
\caption{Stepwise reductions of the bundle $\varphi^*\mathrm{Sp}(2)$. Note that $P$ is defined for any $J$-holomorphic curve while $Q$ and $R$ need the assumption of transversality. $Q$ is not strictly a reduction of $P$ but this is only due to the convention $\omega_1=0$ in \cite{xu2010pseudo} instead of $\omega_2=0$. This does not change any of the results.}
\end{table}
By the uniformisation theorem, any metric on a Riemann surface is conformally equivalent to a constant curvature metric which is being made explicit by the following proposition.
\begin{prop}
\label{conformal flatness}
 The metrics $g_{\mathcal{H}}$ and $g$ are conformally flat for any transverse $J$-holomorphic curve $\varphi\colon X\to \cp$. The conformal factor is given by $\gamma^2=({\alpha_-} {\alpha_+})^{-1/2}$
\end{prop}
\begin{proof}
 Since the metrics $g_{\mathcal{H}}$ and $g$ only differ by the conformal factor $(1+{\alpha^2_-})$ it suffices to prove this statement for $g_{\mathcal{H}}$. 
 First assume that $X$ is simply-connected, i.e. $X\cong\mathbb{D}$ or $\C$.
 In this case, the bundle $R$ admits a global section $s$. Then $s^*\omega_1$ is a unitary $(1,0)$-form on $X$, satisfying the equation
 \[\dd (s^*\omega_1)=s^*(i(\rho_1-\rho_2))\wedge \omega_1=\\d^C(-i/4\log({\alpha_-}{\alpha_+}))\wedge s^*\omega_1.\]
Hence
\[\dd((\gamma^{-1} s^*(\omega_1))=(\dd (\gamma^{-1} )-i \\d^C \log(\gamma^{-1}))\wedge s^*\omega_1=0\]
since $\dd (\gamma^{-1})-i d^C(\gamma^{-1})$ is a $(1,0)$-form. This means that the metric $\gamma^{-2} g_{\mathcal{H}}$ has a closed, unitary $(1,0)$-form $\gamma^{-1} s^* \omega_1$, hence it is flat. 
If $X$ is not simply-connected we can show the statement by passing to the universal cover.
\end{proof}
Putting \cref{laplacealpha}, \cref{laplacebeta}, \cref{derr1} and \cref{derr2} together gives
\begin{align*}
 i \Delta \mathrm{log}({\alpha_-})&=(3{\alpha^2_-}+{\alpha^2_+}-2)\omega_1\wedge \overline{\omega_1} \\
 i \Delta \mathrm{log}({\alpha_+}) &=(3{\alpha^2_+}+{\alpha^2_-}-2) \omega_1\wedge \overline{\omega_1}.
\end{align*}
If we equip $X$ with the metric $g_\mathcal{H}$ then $-\frac{1}{2i} \omega_1\wedge \overline{\omega_1}$ becomes the volume form $\mathrm{vol}_{\mathcal{H}}$ on $X$ and we may rewrite the equations as
\begin{align}
\begin{split}
\label{finalangle}
  \Delta \mathrm{log}({\alpha^2_-})&=-4(3{\alpha^2_-}+{\alpha^2_+}-2) \mathrm{vol}_{\mathcal{H}} \\
  \Delta \mathrm{log}({\alpha^2_+}) &=-4(3{\alpha^2_+}+{\alpha^2_-}-2) \mathrm{vol}_{\mathcal{H}}.
\end{split}
\end{align}
The curvature form on $X$ is then given by 
\begin{align}
\label{curvatureform}
\dd \kappa_{11}=\overline{\tau}\wedge \tau+\omega_1\wedge\overline{\omega_1}-\omega_3\wedge\overline{\omega_3}=(1-{\alpha^2_-}-{\alpha^2_+})\omega_1\wedge \overline{\omega_1}.
\end{align}
Let $\gamma=({\alpha_-}{\alpha_+})^{-1/4}$ as in \cref{conformal flatness}, i.e. $\gamma^{-2} g_\mathcal{H}$ is flat. Denote by $\Delta_0$ the Laplace operator on functions with respect to $\gamma^{-2} g_\mathcal{H}$. 
\begin{thm}
 Let $\varphi\colon X\to \cp$ be a transverse $J$-holomorphic curve. Then the functions ${\alpha_-},{\alpha_+}$ satisfy 
 \begin{align}
 \begin{split}
 \label{finalangleflat}
  \Delta_0 \mathrm{log}({\alpha^2_-})&=-4(3{\alpha^2_-}+{\alpha^2_+}-2)\gamma^2 \\
  \Delta_0 \mathrm{log}({\alpha^2_+}) &=-4(3{\alpha^2_+}+{\alpha^2_-}-2)\gamma^2.
\end{split}
\end{align}
These equations are equivalent to the affine 2D Toda lattice equations for $\mathfrak{sp}(2)$.
The induced Gauß curvature on $X$ is 
\[\frac{2}{\sqrt{{\alpha_-} {\alpha_+}}}(1-{\alpha^2_-}-{\alpha^2_+}).\]
\end{thm}
\begin{proof}
 Eq \ref{finalangleflat} is a direct consequence of the discussion preceding the theorem. 
 Note that the curvature two-form of the metric $g_{\mathcal{H}}$ is equal to $-\frac{1}{2i}\Delta(\mathrm{log}\gamma^2)$. Combining this with \cref{curvatureform} gives 
\[\Delta \mathrm{log}(\gamma^2)=-4(1-{\alpha^2_-}-{\alpha^2_+})\gamma^2\]
which also follows from \cref{finalangleflat}. 
Hence, the induced Gauß curvature on $X$ is 
\[\frac{2}{\sqrt{{\alpha_-} {\alpha_+}}}(1-{\alpha^2_-}-{\alpha^2_+}).\]
If we define ${\hat{\alpha}_{\pm}}=\gamma {\alpha_{\pm}}$ such that $\gamma=(\hat{\alpha}_- \hat{\alpha}_+)^{1/2}$ and \cref{finalangleflat} becomes 
\begin{align}
\begin{split}
 \Delta_0 \mathrm{log}(\hat{\alpha}^2_-)&=-4(2\hat{\alpha}^2_--\hat{\alpha}^{-1}_-\hat{\alpha}^{-1}_+) \\
 \Delta_0 \mathrm{log}(\hat{\alpha}^2_+)&=-4(2\hat{\alpha}^2_+- \hat{\alpha}^{-1}_-\hat{\alpha}^{-1}_+).
\end{split}
\end{align}
So if we let $\hat{\alpha}^2_-=\frac{1}{\sqrt{2}}\exp(\Omega_1)$ and $\hat{\alpha}^2_+=\frac{1}{\sqrt{2}}\exp(-\Omega_2)$ these equations are equivalent to the 2D affine Toda equations for $\mathfrak{sp}(2)$.
\end{proof}
\begin{rmk} 
Note that the result is consistent with \cref{tauprimitive $J$-holomorphic} since a $\tau$-primitive map into $G/T$ is described by the Toda lattice equations for $\mathfrak{g}$ \cite{boltonwoodward}. Furthermore, the relationship between Toda lattice equations and minimal surfaces in $S^4$ has already been observed in \cite{perus1992minimal}.
\end{rmk}
\subsection{The second fundamental form of transverse curves}
As a nearly Kähler manifold, $\cp$ comes equipped with two natural connections. The Levi-Civita connection $\nabla$ and the nearly Kähler connection $\nk$. It turns out that the second fundamental form of a $J$-holomorphic curve is the same for $\nabla$ and $\nk$. Despite $\nk$ having torsion it is the connection that is easier to work with since it preserves the almost complex structure $J$. 
For a fixed $J$-holomorphic curve $\varphi\colon X \to \cp$ consider the map $\Theta_{\varphi}\colon TX\to \nu$ which is defined as $\Theta=-{\alpha^2_-}\mathrm{Id}_{\mathcal{H}}+\mathrm{Id}_{\mathcal{V}}$. Observe that $\Theta$ is injective and let $N_1=\Theta(TX)$. Denote by $N_2$ the orthogonal complement of $N_1$ in $\nu$. It turns out that $N_2$ is in fact equal to the kernel of the orthogonal projection $\nu\to \mathcal{V}_X$. In other words, $\varphi^*(T\cp)$ splits into an orthogonal sum of complex line bundles
\begin{align}
\label{splitting}
 \varphi^*(T\cp)= TX\oplus N_1\oplus N_2
\end{align}
where $N_1\cong TX$ and $N_2\cong (TX)^{-2}$. This splitting is related to the reduction of $\varphi^*(\mathrm{Sp}(2))$ to $Q$. If 
\begin{align}
\label{basechange}
 \begin{pmatrix} u_1 \\ u_2 \\ u_3 \end{pmatrix} =\underbrace{\left(
\begin{array}{ccc}
 \frac{1 }{\sqrt{{\alpha^2_-} +1}} & 0 & \frac{{\alpha_-}}{\sqrt{{\alpha^2_-} +1}} \\
 \frac{-{\alpha_-}}{\sqrt{{\alpha^2_-} +1}} & 0 & \frac{1}{\sqrt{{\alpha^2_-} +1}} \\
 0 & 1 & 0 \\
\end{array}
\right)}_{T^{-1}:=} \begin{pmatrix} \omega_1 \\ \omega_2 \\ \omega_3\end{pmatrix} 
\end{align}
and 
$s$ is a (local) section $s\colon X \to Q$, $s^*(u_1,u_2,u_3)$ is a unitary frame as well. Let $(f_1,f_2,f_3)$ be the dual frame of $s^*(u_1,u_2,u_3)$. Then $f_1$ always takes values in $TX$, $f_2$ in $N_1$ and $f_3$ in $N_2$.
A frame with this property will be called a $Q$-adapted frame from now on.
The connection matrix $A_u$ for the frame $\{f_1,f_2,f_3\}$ is then computed via applying the base change \cref{basechange} to the connection matrix $A_{\omega}$ in \cref{firststructure} and using \cref{dC alpha}
\begin{align}
\label{pullbackconn}
\begin{split}
A_u=T^{-1} A_{\omega} T+ T^{-1} \dd T 
=\left(
\begin{array}{ccc}
 \frac{i \left(\left(2 {\alpha^2_-} -1\right) \rho_1+\rho_2\right)}{{\alpha^2_-} +1} & \frac{2}{{\alpha^2_-} +1} \dd^{0,1}{\alpha_-} & -\frac{\tau ^*}{\sqrt{{\alpha^2_-} +1}} \\
 -\frac{2}{{\alpha^2_-} +1} \dd^{1,0}{\alpha_-} & -\frac{i \left(\left({\alpha^2_-} -2\right) \rho_1-{\alpha^2_-} \rho_2\right)}{{\alpha^2_-} +1} & -\frac{{\alpha_-} \tau ^*}{\sqrt{{\alpha^2_-}+1}} \\
 \frac{\tau }{\sqrt{{\alpha^2_-} +1}} & \frac{{\alpha_-} \tau }{\sqrt{{\alpha^2_-} +1}} & -i (\rho_1+\rho_2) \\
\end{array}
\right).
\end{split}
\end{align} 
\begin{lma}
\label{second ff}
 For a $Q$-adapted frame $f_1,f_2,f_3$ and co-frame $u_1,u_2,u_3$, the second fundamental form of $X$ is equal to 
 \begin{align}
 \label{second ff eqn}
 \sff=\sff_1 \otimes f_2\otimes u_1 +\sff_2 \otimes f_3\otimes u_1
 \end{align}
 for $\sff_1=-\frac{2}{{\alpha^2_-} +1} \dd^{1,0}{\alpha_-}$ and $\sff_2=\frac{\tau}{\sqrt{{\alpha^2_-} +1}}=\frac{{\alpha_+} \omega_1}{\sqrt{{\alpha^2_-} +1}}$.
 We see that transverse points of a $J$-holomorphic curve are never totally geodesic. \todo{What are all totally geodesic curves if we allow superminimallity?}
 Conversely, the frame $\{f_1,f_2,f_3\}$ and the first and second fundamental form of a transverse curve determine ${\alpha_-}$ and ${\alpha_+}$.
\end{lma}
\begin{proof}
 The expression for $\sff$ can be read off from \cref{pullbackconn}. For the last statement, note that $\sff_1=-2\dd^{1,0}\mathrm{arctan}({\alpha_-})$. 
 Assume that ${\alpha_-},{\alpha_+}$ and ${\alpha'_-},{\alpha'_+}$ are two pairs of functions inducing the same first and second fundamental form. Then
 $\mathrm{arctan}({\alpha_-})$ and $\mathrm{arctan}({\alpha'_-})$ differ by a real constant. Equivalently, 
 \[{\alpha'_-}=\frac{{\alpha_-}+C}{1-{\alpha_-} C}\]
 for a constant $C\in \R$.
 Furthermore, the equation for $\sff_2$ implies that 
 \[{\alpha'_+}={\alpha_+} \sqrt{\frac{{\alpha'_-}^2+1}{{\alpha^2_-}+1}}.\]
 Putting all together gives that 
 \[\frac{{\alpha'_-}{\alpha'_+}}{{\alpha_-} {\alpha_+}}=-\frac{({\alpha_-} +C) \sqrt{\frac{C^2+1}{({\alpha_-} C-1)^2}}}{{\alpha_-} ({\alpha_-} C-1)}\]
 which is a constant by \cref{conformal flatness}. However, this is only possible if $C=0$ or if ${\alpha_-}$ is constant. But solutions with ${\alpha_-}$ constant force ${\alpha^2_-}={\alpha^2_+}=1/2$ and hence $C=0$.
\end{proof}
Different $Q$-adapted frames are related by an action of $K\cong S^1$. To work out tensors which are invariant under this action, let $\lambda=(e^{i\theta},e^{3i\theta})$ be an element in $K$ and denote $-2\theta=\vartheta$. The action of $\lambda$ introduces a gauge transformation leading to a transformed set of tensors $(\omega_1',\omega_2',\omega_3',\tau')$ on $Q$.
Note that 
\[(\omega_1',\omega_2',\omega_3')=(e^{i\vartheta} \omega_1,e^{-2 i\vartheta} \omega_2,e^{i\vartheta} \omega_3) \]
which leads to 
\[(u_1',u_2',u_3',\tau')=(e^{i\vartheta} u_1,e^{i\vartheta} u_2,e^{-2 i\vartheta} u_3,e^{-3i\vartheta}\tau)\]
and hence
\[(f_1',f_2',f_3')=(e^{-i\vartheta}f_1,e^{-i\vartheta}f_2,e^{2 i\vartheta}f_3).\]
Consequently, the tensors $\tau\otimes u_{i} \otimes f_3, u_i\otimes f_j, u_3\otimes f_3$ for $i,j=1,2$ are all invariant under $K$ and hence correspond to tensors on $X$. Regard the second fundamental form $\sff$ as a section in $\Omega^{1,0}(X,\mathrm{Hom}(TX,\nu))$. The connection $\nk$ induces connections $\nk^T$ and $\nk^\perp$ on $TX$ and $\nu$ and thus there is a well-defined object $\bar{\partial}_{\nk} \sff=\dd_{\nk} \sff\in \Omega^{1,1}(X,\mathrm{Hom}(TX,\nu))$. We have seen that a transverse $J$-holomorphic curve has no points which are totally geodesic. Having holomorphic second fundamental form is a natural generalisation of being totally geodesic. 
\begin{thm}
 The differential of the second fundamental form is equal to
 \[\dd_{\nk}(\sff)=\frac{2i {\alpha_-}}{{\alpha^2_-}+1}(-2+3 {\alpha^2_-}) \mathrm{vol}_{\mathcal{H}}\otimes u_1\otimes f_2.\]
 In particular, there is no transverse $J$-holomorphic curve with holomorphic second fundamental form.
\end{thm}
\begin{proof}
Using \cref{second ff eqn},
the differential of $\sff$ is computed by
\[\dd_{\nk} \sff= \dd \sff_1 u_1 \otimes f_2-\sff_1\wedge \nk^T u_1 \otimes f_2-u_1\otimes \sff_1\wedge \nk^\perp f_2+ \dd \sff_2 u_1 \otimes f_3-\sff_2\wedge \nk^T u_1 \otimes f_3-u_1\otimes \sff_2\wedge \nk^\perp f_3.\]
This expression takes values in $\Omega^{1,1}(X,TX^{\vee}\otimes \nu)$ and we compute the components in $u_1\otimes f_2$ and $u_1\otimes f_3$ separately.
Note that 
\begin{align*}
 \nk^{T}(u_1)&=-\frac{i \left(\left(2 {\alpha^2_-} -1\right) \rho_1+\rho_2\right)}{{\alpha^2_-} +1} \otimes u_1 \\
 \nk^\perp (f_2)&=-\frac{i \left(\left({\alpha^2_-} -2\right) \rho_1-{\alpha^2_-} \rho_2\right)}{{\alpha^2_-} +1} \otimes f_2 +\frac{{\alpha_-} \tau }{\sqrt{{\alpha^2_-} +1}} \otimes f_3 \\
 \nk^\perp (f_3)&=-\frac{{\alpha_-} \tau ^*}{\sqrt{{\alpha^2_-}+1}} \otimes f_2 -i (\rho_1+\rho_2) \otimes f_3
\end{align*}
and
\begin{align*}
 \dd \sff_1&=i \frac{1-{\alpha^2_-}}{(1+{\alpha^2_-})^2} \dd {\alpha_-} \wedge \dd^C \mathrm{log}({\alpha_-})-\frac{i{\alpha_-}}{1+{\alpha^2_-}}\Delta \mathrm{log}({\alpha_-})\\
 &=2i \frac{1-{\alpha^2_-}}{(1+{\alpha^2_-})^2} \dd^{1,0} {\alpha_-} \wedge \dd^C \mathrm{log}({\alpha_-})+\frac{{\alpha_-}}{1+{\alpha^2_-}}(-3{\alpha^2_-}-{\alpha^2_+}+2)\omega_1\wedge \overline{\omega_1}.
\end{align*}
Furthermore, the $u_1\otimes f_2$ component of $-\sff_1\wedge \nk^T u_1 \otimes f_2-u_1\otimes \sff_1\wedge \nk^\perp f_2$ is equal to
\begin{align*}
 \frac{\sff_1 i}{{\alpha^2_-}+1}((2{\alpha^2_-}-1)\rho_1+\rho_2+({\alpha^2_-}-2)\rho_1-{\alpha^2_-}\rho_2)=\frac{i({\alpha^2_-}-1)}{{\alpha^2_-}+1}\sff_1\wedge (3\rho_1-\rho_2)=i\frac{1-{\alpha^2_-}}{1+{\alpha^2_-}}\sff_1 \wedge \dd^C\log({\alpha_-}).
\end{align*}
Finally, the $u_1\otimes f_2$ component of $-u_1\otimes \sff_2\wedge \nk^\perp f_3$ is equal to
\begin{align}
 \sff_2\wedge \frac{{\alpha_-} \bar{\tau}}{\sqrt{{\alpha^2_-}+1}}=\frac{{\alpha_-} {\alpha^2_+}}{{\alpha^2_-}+1}\omega_1 \wedge \overline{\omega_1}.
\end{align}
Observe that various terms cancel and that the $u_1 \otimes f_2$ component of $\dd_{\nk} \sff$ is
\[\frac{{\alpha_-}}{{\alpha^2_-}+1}(-3{\alpha^2_-}+2)\omega_1 \wedge \overline{\omega_1}=\frac{2i {\alpha_-}}{{\alpha^2_-}+1}(3{\alpha^2_-}-2)\omega_1 \wedge \overline{\omega_1}.\]
The $u_1\otimes f_3$ component is computed in an analogous way and equal to
\begin{align}
 \frac{{\alpha_-}}{{\alpha^2_-}+1} \sff_2\wedge(\dd {\alpha_-} -i\dd^C {\alpha_-})=0
\end{align}
since both $\sff_2$ and $\dd {\alpha_-} -i\dd^C {\alpha_-}$ are $(1,0)$ forms. This proves the formula
 \[\dd_{\nk}(\sff)=\frac{2i {\alpha_-}}{{\alpha^2_-}+1}(-2+3 {\alpha^2_-}) \mathrm{vol}_{\mathcal{H}}\otimes u_1\otimes f_2.\]
For a transverse $J$-holomorphic curve, ${\alpha_-}$ is always positive. So the second fundamental form is holomorphic if and only if ${\alpha_-}$ is constant to $\sqrt{2/3}$. However, no such solution exists for \cref{finalangleflat}.
\end{proof}
\subsection{Holomorphic structure of the normal bundle}
The nearly Kähler connection $\nk$ preserves $J$ and $\nu$ is a complex subbundle of $T\cp$. Hence $\nk^\perp$ gives rise to a holomorphic structure on $\nu$. Let $\varphi\colon X \to \cp$ be a transverse torus. The degree of $\nu$ is zero since the first Chern class of any nearly Kähler manifold vanishes, i.e.
\[c_1(\nu)=c_1(T\cp)-c_1(TX)=0.\]
Let $\mathrm{Bun}(r,d)$ be the space of indecomposable holomorphic bundles of rank $r$ and degree $d$ over $X$.
By Atiyah's classification of holomorphic vector bundles over elliptic curves \cite{atiyah1957vector}, $\mathrm{Bun}(2,0)$ is isomorphic to a two-torus. For any element $E\in \mathrm{Bun}(2,0)$ the line bundle $\Lambda^2(E)$ is trivial. This is consistent with the fact that $\Lambda^2(\nu)=TX$. The space $\mathrm{Bun}(2,0)$ has a distinguished element $E_0$, the unique non-trivial extension of the sequence 
\[0 \to \C \to E \to \C \to 0.\]
Based on this, there are a-priori three possibilities for $\nu$. Either $\nu$ is decomposable, isomorphic to $E_0$ or another element in $\mathrm{Bun}(2,0)$. In the following we will see that for a transverse $J$-holomorphic torus in $\cp$, $\nu$ is always isomorphic to $E_0$.
Let $\sigma_{ij}$ be the components of the connection matrix $A_u$. Assume that the frame $\{f_1,f_2,f_3\}$ is $R$-adapted, such that \cref{R equations} hold. Then $s=s_2 f_2+s_3f_3$ describes a general section in the normal bundle for $s_2,s_3\in \Omega(X,\C)$. 
By the Leibniz rule for $\bar{\partial}$, holomorphic sections are solution of the equation 
\begin{align}
\label{first hol eqn}
 s_2\sigma_{32}(\frac{\partial}{\partial \bar{z}})+s_3\sigma_{33}(\frac{\partial}{\partial \bar{z}})+\frac{\partial}{\partial \bar{z}}(s_3)&=0 \\
 \label{second hol eqn}
 s_2\sigma_{22}(\frac{\partial}{\partial \bar{z}})+s_3\sigma_{23}(\frac{\partial}{\partial \bar{z}})+\frac{\partial}{\partial \bar{z}}(s_2)&=0.
\end{align}
Since $\sigma_{32}$ is of type $(1,0)$ it annihilates $\frac{\partial}{\partial \bar{z}}$. Note that on $X$, $u_1=\sqrt{1+{\alpha^2_-}}\omega_1$ and hence by \cref{conformal flatness} we can find a local coordinate $z$ on $X$ such that
\[\dd z=\frac{({\alpha_-} {\alpha_+})^{1/4}}{\sqrt{1+{\alpha^2_-}}} u_1.\]
Then \cref{first hol eqn} reduces to 
\[ s_3\sigma_{33}(\frac{\partial}{\partial \bar{z}})+\frac{\partial}{\partial \bar{z}}(s_3)=0\]
and 
\[\sigma_{33}(\frac{\partial}{\partial \bar{z}})=\frac{\partial}{\partial \bar{z}} \log({\alpha^{-1/2}_+} {\alpha^{1/2}_-}).\]
Hence, all solutions of \cref{first hol eqn} are given by
\[s_3=c {\alpha^{-1/2}_-} {\alpha^{1/2}_+}\]
for a constant $c\in \C$.
Furthermore, we have
\begin{align*}
\sigma_{23}(\frac{\partial}{\partial \bar{z}})&=\frac{-({\alpha_-} {\alpha_+})^{3/4}}{\sqrt{1+{\alpha^2_-}}} \\
\sigma_{22}(\frac{\partial}{\partial \bar{z}})&=\frac{\partial}{\partial \bar{z}}(\log({\alpha^{1/4}_+}{\alpha^{-3/4}_-}\sqrt{1+{\alpha^2_-}})).
\end{align*}
So, define $s_2'=s_2 {\alpha^{1/4}_+}{\alpha^{-3/4}_-}\sqrt{1+{\alpha^2_-}}$ which makes \cref{second hol eqn} equivalent to 
\[s_3 \frac{-({\alpha_-} {\alpha_+})^{3/4}}{\sqrt{1+{\alpha^2_-}}}+{\alpha_+}^{-1/4}{\alpha_-}^{+3/4}(1+{\alpha^2_-})^{-1/2}\frac{\partial}{\partial \bar{z}}s_2'=0 \]
and this in turn results in 
\[\frac{\partial}{\partial \bar{z}}s_2'=c {\alpha_+}^{3/2}{\alpha^{-1/2}_-}.\]
But the function ${\alpha^{-1/2}_-} {\alpha^{1/2}_+}$ is non-vanishing which forces $c=0$ and $s_2'$ to be a constant. This computation implies the following two lemmas.
\begin{lma}
\label{hol normal 1}
 Let $X$ be a transverse torus then $h^0(X,\nu)=1$
\end{lma}
 \begin{lma}
\label{hol normal 2}
 The holomorphic section $s_2'f_2$ induces an exact sequence of holomorphic bundles
 \[ 0 \to \C \to \nu \to \C \to 0.\]
 \end{lma}
 \begin{proof}
 Since $s_2$ is non-vanishing the section gives an injection $\C\to \nu$. The quotient bundle admits a non-vanishing holomorphic section because $\sigma^{0,1}_{33}$ is $\bar{\partial}$-exact. Hence the quotient is trivial. 
 \end{proof}
Summarising the two statements above gives 
\begin{prop}
 As a holomorphic bundle, the normal bundle $\nu$ of a transverse torus is the unique non-trivial extension of the sequence 
 \[ 0 \to \C \to \nu \to \C \to 0.\]
\end{prop}
\section{A Bonnet-type theorem}
\label{section bonnet}
The uniqueness part of the classical Bonnet theorem says that two surfaces $\Sigma,\Sigma'\subset \R^3$ with the same first and second fundamental form $I$ and $\sff$ necessarily differ by an isometry of $\R^3$. The existence part states that given a simply-connected surface $\Sigma$ with the tensors $I,\sff$ defined on $\Sigma$ satisfying the Gauss and Codazzi equation there is an immersion $\Sigma \to \R^3$ with induced first and second fundamental form equal to $I$ and $\sff$. One way how to prove this statement is via the theorem of Maurer-Cartan:
 Let $G$ be a Lie Group with Maurer-Cartan form $\Omega_{MC}$, let $N$ be connected and simply-connected, equipped with $\eta\in \Omega^1(N,\mathfrak{g})$ satisfying
 $\dd \eta+\frac{1}{2}[\eta,\eta]=0.$
 Then there exists a smooth map $f\colon N\to G$, unique up to left translation in $G$, such that $f^*\Omega_{MC}=\eta$. 
In this section we will show, also using the theorem of Maurer-Cartan, an analogue of Bonnet's theorem for $J$-holomorphic curves in $\cp$. \par
We have seen that given a transverse $J$-holomorphic curve $\varphi \colon X\to \cp$ there are functions ${\alpha_-},{\alpha_+},\gamma$ that satisfy \cref{finalangleflat}, where $\gamma$ is determined by ${\alpha_-}$ and ${\alpha_+}$. Up to a constant factor, the functions $({\alpha_-},{\alpha_+},\gamma)$ determine the first and second fundamental form of $X$.
Composing $\varphi$ with an element in an automorphism in $\spp$ leaves the quantities $({\alpha_-},{\alpha_+})$ invariant since they are defined via components of $\Omega_{MC}$. Besides, the data $({\alpha_-},{\alpha_+})$ determines the first and second fundamental form up to a constant. \par
If a Bonnet-theorem holds for $J$-holomorphic curves in $\cp$ then the first and second fundamental form determine the curve up to isometries. Furthermore it would say that such a curve exists if the Gauss and Codazzi equations are satisfied. In our setting, the system \cref{finalangle} plays the role of these equations. This raises the following question. \textit{Are the solutions of \cref{finalangle} in one to one correspondence to transverse $J$-holomorphic curves up to isometries?}
Later, we will see that the answer to this question is no because there are periodic solutions of \cref{finalangle} which do not descend to a two-torus. However, there is a positive result for when $X$ is simply-connected. 
\begin{lma}
\label{Bonnetlemma}
 Let $X$ be a simply-connected Riemann surface equipped with a metric $k$. Let furthermore ${\alpha_-},{\alpha_+} \colon X\to \R^{>0}$ such that \cref{finalangleflat} are satisfied for $\gamma=({\alpha_-} {\alpha_+})^{-1/4}$ and that $\gamma^{-2}k$ is flat. Then there is a $J$-holomorphic immersion $\varphi\colon X \to \cp$ such that the induced metric satisfies $g_{\mathcal{H}}=k$. The angle functions of $\varphi$ are exactly ${\alpha_-}$ and ${\alpha_+}$. The immersion is unique up to isometries of $\cp$ and an element in $S^1/{\Z_4}\cong S^1$ which parametrises a choice of a unitary $(1,0)$-form $\omega_0$ on $X$ such that $\dd(\gamma^{-1} \omega_0)=0$.
\end{lma}
\begin{proof}
 Since $X$ is simply-connected and conformally flat, it is isomorphic to $\C$ as a complex manifold by the uniformisation theorem. In particular, $\Omega^{1,0}(X)$ is trivial as a bundle, let $\omega_0$ be a unitary $(1,0)$-form on $X$.
 Arguing similarly as in the proof of \cref{conformal flatness} we see that, since $\gamma^{-2}k$ is flat,
 \begin{align}
 \label{domega0}
 \dd \omega_0=i \dd^C \log(\gamma) \wedge \omega_0.
 \end{align}
 Now, define a $\mathfrak{sp}(2)$-valued one-form on $X$ by
 \[\eta_{\omega_0}=\begin{pmatrix} \frac{i}{8}(-3 \dd^C \mathrm{log}({\alpha_-})+\dd^C \mathrm{log}({\alpha_+}))+j{\alpha_-} \overline{\omega_0} & -\frac{\omega_0}{\sqrt{2}} \\ \frac{\omega_0}{\sqrt{2}} & \frac{i}{8}(-\dd^C \mathrm{log}({\alpha_-})+3 \dd^C \mathrm{log}({\alpha_+}))+j{\alpha_+} \omega_0.
\end{pmatrix}.\]
And observe that \[\dd \eta_{\omega_0}+\frac{1}{2}[\eta_{\omega_0},\eta_{\omega_0}]=0\] is equivalent to \cref{finalangleflat} and \cref{domega0}. Hence, by Cartan's theorem, there is an immersion $\Phi\colon X\to \mathrm{Sp}(2)$ such that $\Phi^*(\Omega_{MC})=\eta$ which is unique up to left multiplication in $\mathrm{Sp}(2)$. 
Note that $\varphi^*(\Omega_{MC})=\eta$ is equivalent to the equations
 \begin{align}
 \begin{split}
 \label{pullbackequations}
 \omega_0 &=\Phi^*(\omega_1) \qquad \dd^C(\log ({\alpha_-}))=-\Phi^*(3\rho_1-\rho_2) \qquad \omega_0={\alpha_-} \Phi^*(\omega_3)\\
 0 &=\Phi^*(\omega_2) \qquad \dd^C(\log ({\alpha_+}))=-\Phi^*(\rho_1-3\rho_2) \qquad \omega_0={\alpha_+} \Phi^*\tau.
 \end{split}
 \end{align}
Consider the map $\varphi=\pi\circ \Phi\colon X\to \cp$
\[
\begin{tikzcd}
                     & \mathrm{Sp}(2) \arrow[d, "\pi"] \arrow[d] \\
X \arrow[r, "\varphi"] \arrow[ru, "\Phi"] & \mathbb{CP}^3              
\end{tikzcd}.
\]
Then $\varphi$ is also an immersion because $v\in \mathrm{ker}(\dd \varphi)\subset T^{1,0}(X)$ is equivalent to 
$\Phi^*\omega_i(v)=\omega_i(\dd \Phi (v))=0$ for $i=1,2,3$. 
By \cref{pullbackequations}, this implies that $\omega_0(v)=0$ and hence $v=0$. Furthermore, $\varphi$ is $J$-holomorphic since $\Phi$ pulls back the forms $\omega_i$ to multiples of $\omega_0$. Since $\omega_1$ is unitary for $g_{\mathcal{H}}$, $\omega_0$ for $k$ and $\Phi^*(\omega_1)=\omega_0$ the metric $g_{\mathcal{H}}$ induced by $\varphi$ is equal to $k$. Since left-multiplication on $\mathrm{Sp}(2)$ acts on $\cp$ by isometries it remains to prove that choosing $e^{i\vartheta} \omega_0$ as unitary $(1,0)$-form on $X$ yields, up to isometries, the same immersion $\varphi \colon X\to \cp$ as $\omega_0$ if $e^{4i\vartheta}=1$. 
Let $R_\lambda$ be the right multiplication of an element $\lambda=(e^{i\vartheta},e^{i\vartheta})$ on $\spp$. Then $R_{\lambda}^*(\Omega_{MC})=\mathrm{Ad}_{\lambda^{-1}}(\Omega_{MC})$. From our knowledge of the adjoint action of $K$ on $\mathfrak{sp}(2)$ we infer that 
\[R_{\lambda}^*(\omega_1,\omega_2,\omega_3,\tau,\rho_1,\rho_2)=(e^{-2i\vartheta}\omega_1,e^{4i\vartheta}\omega_2,e^{-2i\vartheta}\omega_3,e^{6i \vartheta}\tau,\rho_1,\rho_2).\]
Hence, if $e^{4i\vartheta}=1$ then $\omega_1$ and $\tau$ transform in the same way. This means that in this case if $\varphi$ satisfies $\varphi^*{\Omega_{MC}}=\eta_{\omega_0}$ then $(\varphi\circ R_{\lambda})^*(\Omega_{MC})=\eta_{e^{i\vartheta}\omega_0}$. But right multiplication on $\mathrm{Sp}(2)$ does not affect the immersion $\varphi\colon X \to \cp$.
\end{proof}
Note that \cref{Bonnetlemma} as it is stated requires $X$ to be equipped with a fixed metric. However, it is more natural to let metric be one of the quantities to be determined, such as ${\alpha_-}$ and ${\alpha_+}$. Think of $X$ as equipped with a flat metric and consider a solution of \cref{finalangleflat} with $\gamma=({\alpha_-}{\alpha_+})^{-1/4}$. Then $({\alpha_-},{\alpha_+})$ is a solution of \cref{finalangle} with the metric $g=\gamma^2 g_0$. However, we could also have chosen the metric $\lambda^2 g_0$ for a constant $\lambda>0$. In general, \cref{Bonnetlemma} will produce different immersion $\varphi_{\lambda} \colon X\to \cp$. The following theorem is then a reformulation of \cref{Bonnetlemma}.
\begin{thm}
\label{Bonnet}
 Let $X$ be a simply-connected Riemann surface and let $({\alpha_-},{\alpha_+})$ be solutions of \cref{finalangleflat} for some flat metric on $X$. Then there is a $J$-holomorphic immersion $\varphi\colon X \to \cp$ such that the angle functions $\varphi$ are exactly ${\alpha_-}$ and ${\alpha_+}$. The immersion is unique up to isometries of $\cp$, a choice $\lambda>0$ and an element in $S^1/{\Z_4}\cong S^1$ which parametrises a choice of a unitary $(1,0)$-form $\omega_0$ on $X$ such that $\dd(\gamma^{-1} \omega_0)=0$.
\end{thm}
In other words, a solution $({\alpha_-},{\alpha_+})$ of \cref{finalangleflat} specifies the immersion up to isometries and a choice of a constant $(\lambda, \omega_0)\in \C^*$.
In the case when the solutions $({\alpha_-},{\alpha_+})$ and hence $g$ have symmetries, the different embeddings $\varphi_{\lambda, \gamma}$ might come from reparametrisations of $X$. Consider for example, $X=\C$ and let $L_\lambda\colon \C \to \C$ be the multiplication by $\lambda$. Define furthermore ${\alpha_-}_{\lambda}={\alpha_-} \circ L_{\lambda},{\alpha_+}_{\lambda}={\alpha_+} \circ L_{\lambda},g_{\lambda}=g\circ L_{\lambda}$. If $({\alpha_-},{\alpha_+})$ solves \cref{finalangle} for a metric $g$ on $\C$ then $({\alpha_-}_\lambda,{\alpha_+}_\lambda)$ is a solution of \cref{finalangle} for the metric $\lambda^{2} g_{\lambda}$. This solution comes from the $J$-holomorphic curve $\varphi \circ L_{\lambda}$. In particular, if $g_\lambda=g$ then the different immersions $\varphi_{\lambda}$ come from reparametrisations of $X$ by $L_\lambda$. Similarly, if the induced $g$ has radial symmetry, different choices of $\omega_0$ correspond to reparametrisations by rotations. 
\subsection{Special solutions of the Toda equations}
Note that \cref{finalangleflat} are symmetric in ${\alpha_-}$ and ${\alpha_+}$, meaning a distinguished set of solutions is of the form ${\alpha_-}={\alpha_+}$. This reduces \cref{finalangleflat} to
\[\Delta_0 \Omega=-8\sqrt{2} \mathrm{sinh}(\Omega)\]
which becomes the Sinh-Gordon equation after rescaling the metric with a constant factor. This is somewhat similar to the situation in $S^3\times S^3$ where $J$-holomorphic curves with $\Lambda=0$ are locally described by the same equation.
In particular,
\begin{prop}
 Transverse $J$-holomorphic curves in $\cp$ with ${\alpha_-}={\alpha_+}$ are locally described by the same equation as constant mean curvature tori in $\R^3$.
\end{prop}
Geometrically, the condition ${\alpha_-}={\alpha_+}$ is satisfied for a $J$-holomorphic curve if the corresponding minimal surface $X\to S^4$ lies in a totally geodesic $S^3$, see \cref{twistor lifts totally geodesic}.
Hence the twistor lifts of the surfaces $T_{k,m}\subset S^3\subset S^4$ from \cref{minimal s3} give rise to examples of $J$-holomorphic tori with ${\alpha_-}={\alpha_+}$. The Clifford torus $T_{1,1}$ plays a special because it is a $\mathbb{T}^2$ group orbit. We refer to Clifford tori to all $J$-holomorphic curves isometric to the twistor lift of $T_{1,1}$. Since isometries leave ${\alpha_-}$ and ${\alpha_+}$ invariant this means that ${\alpha_-}$ and ${\alpha_+}$ are constant on Clifford tori. Observe, that there is in fact only one solution for ${\alpha_-}$ and ${\alpha_+}$ constant.
\begin{lma}
\label{constantangle}
 If either ${\alpha_-}$ or ${\alpha_+}$ is constant then both must be constant and equal to $1/\sqrt{2}$.
\end{lma}
\begin{thm}
\label{classflatcurves}
 Let $\varphi\colon X \to \cp$ be a transverse $J$-holomorphic curve such that the induced metric $g_{\mathcal{H}}$ is flat and $X$ is equal to $\mathbb{T}^2$ or $\C$ as a complex manifold. Then $\varphi$ parametrises a Clifford torus for $X=\mathbb{T}^2$ and its universal cover if $X=\C$.
\end{thm}
\begin{proof}
By passing to the universal cover it suffices to assume $X=\C$. Furthermore, 
by \cref{curvatureform}, flatness of $g_\mathcal{H}$ is equivalent to $1={\alpha^2_-}+{\alpha^2_+}$ and so \cref{finalangleflat} implies that $\log({\alpha^2_-}(1-{\alpha^2_-}))$ is harmonic. It is also bounded because ${\alpha^2_-},{\alpha^2_+}>0$ and hence constant. Since ${\alpha_-}$ and ${\alpha_+}$ must be constant on Clifford tori It follows from \cref{constantangle} that ${\alpha^2_-}={\alpha^2_+}=1/2$. Now, \cref{Bonnet} can be applied to prove uniqueness. Note that since $\C$ carries the flat metric, a different choice of $(\omega_0,\lambda)$ amounts to applying an isometry of $X$.
\end{proof}
\todo{How about flat embeddings of $\mathbb{D}$?}
\section{$\mathrm{U}(1)$-invariant $J$-holomorphic curves}
\label{u1 section}
$\uo$-actions on the nearly Kähler $\cp$ have been studied on the resulting $G_2$-cone by Atiyah and Witten in the context of dimensional reductions of M-theory \cite{atiyah2001m}. Closed expressions for the induced metric, curvature and the symplectic structure on $\R^6$ have recently been found by Acharya, Bryant and Salamon \cite{acharya2020circle} for one particular $\mathrm{U}(1)$-action. \par
While superminimal curves can be parametrised very explicitly our description of transverse curves has been, with the exception of Clifford tori, rather indirect so far. Imposing $\uo$-symmetry on the curves reduces a system of PDE's to a system of ODE's. In terms of ${\alpha_-}$ and ${\alpha_+}$, the 2D Toda lattice equation will reduce to the 1D Toda lattice equation. Killing vector fields on $\cp$ are in one to correspondence with Killing vector fields on $S^4$. We will provide a twistor perspective on the work of $\uo$-invariant minimal surfaces \cite{ferus1990s}. The geodesic equation on $S^3$ for the Hsiang-Lawson metric is replaced by the computationally less involved 1D Toda equation. The twistor perspective establishes relationship between $\uo$-invariant curves and the toric nearly Kähler geometry of $\cp$.
\subsection{$\uo$-invariant minimal surfaces in $S^3$}
\label{minimal s3}
In \cite{lawson1970complete}, B. Lawson developed a rich theory of minimal surfaces in $S^3$. For coprime positive integers $m\geq k$ he finds the minimal surface
\[T_{k,m}=\{(z,w)\in S^3\subset \C^2 \mid \mathrm{Im}(z^m\bar{w}^k)=0\}\]
 which represent Klein bottles for $2|(mk)$ and tori otherwise. The surface $T_{1,1}$ is the Clifford torus $T$. They are geodesically ruled and are invariant under the $\uo$-action $\rho$ given by 
 \[e^{i \vartheta}(z,w)=(e^{k i \vartheta}z,e^{m i \vartheta}z).\]
 It is furthermore shown that closed minimal surfaces invariant under this action are in one to one correspondence with closed geodesics in the orbit space $S^3/\uo$ equipped with an ovaloid metric. By studying this metric explicitly, a rationality condition on the initial values for the geodesic to be closed is obtained. This gives rise to a countable family of tori $(T_{k,m,a})_{a\in A_{k,m}}$ where $A_{k,m}$ is a certain countable dense subset of $(0,\pi/2)$. This family is bounded by $T_{k,m}$ corresponding to the boundary case $a=0$ and the Clifford torus $T$ for $a=\pi/2$. Similarly, for either $m=0,k=0$ or $m=\pm k$ there is a countable family $C_a$ of minimal tori bounded by the Clifford torus and a totally geodesic two-sphere \cite[Theorem 8]{hsiang1971minimal}.\par
 Any cohomogeneity-one minimal surface in $S^3$ belongs to the families $C_a$ or $T_{k,m,a}$ \cite[Theorem 9]{hsiang1971minimal}.
 Furthermore, Lawson constructs for any minimal surface in $S^3$ an associated minimal surface in $S^5$ which he calls bipolar surface. For the family $T_{k,m}$ the bipolar surfaces lie in $S^4$ but not in a totally geodesic $S^3$.
Extend $\rho$ to an action on $S^4\subset \C^2\oplus \R$ by letting $\uo$ act trivially on the $\R$ component.
 The Hsiang-Lawson family $T_{k,m,a}$ is contained in the totally geodesic three sphere $x=0$. Furthermore, the corresponding surfaces of this family are also $\uo$-invariant but under the modified action $(\tilde k, \tilde m)=(m-k,m+k)$. 
 Denote by $\widetilde{T}_{k,m,a}$ the family of bipolar surfaces constructed from $T_{\tilde{k},\tilde{m},a}$. The two families $T_{k,m,a}$ and $\widetilde{T}_{k,m,a}$ will play a special role in the following discussion. 
\subsection{Killing vector fields}
Since the nearly Kähler $\cp$ has isometry group $\mathrm{Sp}(2)$ any element $\xi \in \mathfrak{sp}(2)$ gives rise to the Killing vector field
\[K^{\xi}(x)=\frac{\dd}{\dd t} \exp(t \xi) x.\]
Assume that there is $t>0$ such that $\exp(t \xi)$ equals the identity $e\in \mathrm{Sp}(2)$, i.e. $\xi$ corresponds to an action $\rho$ of $\mathrm{U}(1)$ on $\cp$. Acting via $\rho$ on integral curves of $JK^{\xi}$ gives rise to $\uo$-invariant $J$-holomorphic curves in $\cp$. Another way of stating this is that $[K^{\xi},JK^{\xi}]=0$ and that the span of $K^{\xi}$ and $JK^{\xi}$ defines an integrable distribution $V_{\xi}$ on $M=\cp\setminus ({K^{\xi}})^{-1}(0)$. The integral submanifold are exactly the $\rho$ invariant $J$-holomorphic curves in $\cp$ which foliate $M$. Due to the isomorphism between $\mathfrak{sp}(2)\cong\mathfrak{so}(5)$, the element $\xi$ also gives rise to a Killing vector field $K^{\xi}_{S^4}$. This results in a uniqueness statement for minimal surfaces in $S^4$. 
\begin{lma}
\label{uniqueness u1 maps}
 Let $x\in S^4$ such that $V \subset T_x S^4$ be a two-dimensional subspace containing $K_{S^4}^{\xi}(x)$ and let $p \in \pi^{-1}(x)\subset \cp$ be the twistor lift of $V$. If $K_p^{\xi}$ is non-vertical there is a locally unique minimal surface $\Sigma$ with $x\in \Sigma$ and $V=T_x\Sigma$.
\end{lma}
Let $X$ be an integral submanifold of the distribution $V_{\xi}$. Since it is a $J$-holomorphic curve the bundle $\mathrm{Sp}(2)|_X$ reduces to the $\mathbb{Z}_8$ bundle $R_X$, see \cref{summaryreducedbundle}.
There is an $\mathbb{Z}_8$ bundle $R$ over $M$ which restricts to $R_X$ on each integral submanifold. Note that the construction of $R$ is linked to the subspace $\mathfrak{r}\subset \mathfrak{sp}(2)$ which becomes apparent in the following lemma.
\begin{lma}
\label{R lemma}
 If a section $s\colon \cp \supset U \to \mathrm{Sp}(2)$ lies in $R$ then $s^*(\Omega_{MC})(K^{\xi})=s^{-1} \xi s \in \mathfrak{r}$.
\end{lma}
\begin{proof}
 Let $s$ be a section of $R$. Consider the section $s'=L_g \circ s \circ L_{g^{-1}}$ where $g=\exp(t \xi)$ for some $t\in \R$. Then $s'^*(\Omega_{MC})(K_x^{\xi})=\Omega_{MC}(K_{gx}^{\xi})\in \mathfrak{r}$. In other words, $s'$ satisfies \cref{R equations}
 and hence $s'$ also has values in $R$. Since the sections $s$ and $s'$ are joined by a continuous path and $R$ has discrete structure group it follows that $s'=s$, i.e.
 \[s(g x)=g s(x).\]
 This implies
 \[s^*(\Omega_{MC})(K^{\xi})=s^{-1}\dd s(K^{\xi})=s^{-1} \xi s.\]
\end{proof}
The point of the previous lemma is that it gives a more explicit description of the bundle $R$. Over $U$, define the $K_F$ bundle $R'=\{g\in \mathrm{Sp}(2) \mid g^{-1} \xi g \in \mathfrak{r}\}$. By \cref{R lemma}, $R=R'$. The point of the previous lemma is that it gives a more explicit description of the bundle $R$.
\begin{crl}
 A section $s\colon U \to \mathrm{Sp}(2)$ takes values in $R$ if and only if $s^{-1} \xi s \in \mathfrak{r}$.
\end{crl}
The presence of the vector field $K^{\xi}$ means we can define the functions $h=\|K^{\xi}\|_{\mathcal{H}}^2$, $v_-=\|K^{\xi}\|_{\mathcal{V}}^2$. Furthermore, if we restrict $\tau$ to $R$ the quantity $v_+=|\tau(K^{\xi})|^2$ is well-defined. This gives
\begin{align}
\label{vminus vplus defintion}
{\alpha_-}=\sqrt{v_-/h}, \qquad {\alpha_+}=\sqrt{v_+/h}.
\end{align}
Let $X$ be a transverse integral submanifold. Since $K^{\xi}$ and $JK^{\xi}$ commute, these vector fields give rise to coordinates $(u,t)$ such that $\frac{\partial}{\partial u}=K^{\xi}$ and $\frac{\partial}{\partial t}=JK^{\xi}$. The induced metric $g_{\mathcal{H}}$ on $X$ is equal to 
$h (\dd u^2+\dd t^2)$. In particular, due to \cref{conformal flatness} and since $h$ is the conformal factor of the metric, the quantity $C=h\sqrt{{\alpha_-} {\alpha_+}}=h^{1/2}(v_-v_+)^{1/4}$ is constant along $X$. Hence, \cref{finalangleflat} reduces to 
\begin{align}
\begin{split}
\label{odesystem}
  \frac{\dd^2}{\dd t^2} \mathrm{log}(v_-)&=4(\frac{C^2}{\sqrt{v_-v_+}}-2v_-) \\
  \frac{\dd^2}{\dd t^2} \mathrm{log}(v_+) &=4(\frac{C^2}{\sqrt{v_-v_+}}-2v_+).
\end{split}
\end{align}
It is clear that the equations of \cref{summaryreducedbundle} hold when the forms are restricted to $R$. However, the fact that $K^{\xi}$ is a Killing vector field guarantees that the following additional equations are satisfied
\begin{align}
\label{additionaleqns 1}
 \rho_1(JK^{\xi})&=\rho_2(JK^{\xi})=0 \\
 \label{additionaleqns 2}
 \omega_1(K^{\xi})&=\lambda \sqrt{h} \\
 \label{additionaleqns 3}
 \omega_3(K^{\xi})&=\lambda \sqrt{v_-} \\
 \label{additionaleqns 4}
 \tau(K^{\xi})  &=\lambda \sqrt{v_+}.
\end{align}
for a constant $\lambda\in S^1$. Note that \cref{additionaleqns 1} are satisfied because isometries preserve the quantities ${\alpha_-}$ and ${\alpha_+}$. \cref{additionaleqns 2} holds because $(u,t)$ are isothermal coordinates for $g_{\mathcal{H}}$ on integral submanifold. Lastly, \cref{additionaleqns 3} and \cref{additionaleqns 4} follow from \cref{additionaleqns 2} and the definition of $v_-,v_+$, i.e. \cref{vminus vplus defintion} as well as the equations defining $R$, i.e. \cref{R equations}. Note that we have abused notation slightly here. By how $K_F$ acts on $\mathfrak{r}$ the forms $\omega_1,\dots,\omega_3$ are basic on $R$ but take values in a possibly non-trivial line bundle on $X$ whose structure group reduces to $K_F$. In other words, $\lambda$ is only well-defined up to multiplication of a fourth root of unity. However, this poses no problem since the results which follow will only depend on $\mu=\lambda^4$.
\subsection{Lax Representation and Toda Lattices}
Since a general transverse curve satisfies the 2D periodic Toda lattice equation, $\uo$-invariant curves will satisfy the 1D periodic Toda lattice equations for the same Lie algebra, i.e. \cref{odesystem}. These equations admit a Lax representation \cite{bogoyavlensky1976perturbations}. In the following, we will work out this Lax representation directly from the formalism of adapted frames.
Consider the restriction of $\Omega_{MC}$ to the reduced bundle $R$, which will still be denoted by $\Omega=\Omega_{MC}$. We have
\[\dd\Omega=-\frac{1}{2}[\Omega,\Omega].\]
Note that \cref{additionaleqns 1}-\cref{additionaleqns 4} imply
\begin{align}
\label{killingeqn}
\Omega(JK^{\xi})&=\begin{pmatrix} -ji \bar{\lambda} \sqrt{v_-} & i\frac{\overline{\lambda} \sqrt{h}}{\sqrt{2}} \\ i\frac{\lambda \sqrt{h}}{\sqrt{2}} & j i \sqrt{v_+} \lambda
        \end{pmatrix}\\
\Omega(K^{\xi})&=\begin{pmatrix} \frac{-i}{4}(\frac{\dd}{\dd t} \log(v_+))+               
j \bar{\lambda} \sqrt{v_-} & -\frac{\overline{\lambda} \sqrt{h}}{\sqrt{2}} \\ \frac{\lambda \sqrt{h}}{\sqrt{2}} & \frac{i}{4}( \frac{\dd}{\dd t} \log(v_+)) +j \lambda \sqrt{v_+} 
        \end{pmatrix}.
\end{align}
Consequently, $K^{\xi}(\Omega(JK^{\xi}))=0$ because $h,v_-,v_+$ are constant along $K^{\xi}$.     
With this in mind, let us evaluate both two-forms at $K^{\xi}\wedge JK^{\xi}$
\begin{align}
\label{laxgeometric}
\frac{\dd}{\dd t}(\Omega(K^{\xi}))=(JK^{\xi})(\Omega(K^{\xi}))=-\dd\Omega(K^{\xi},JK^{\xi})=\frac{1}{2}[\Omega,\Omega](K^{\xi},JK^{\xi})=[\Omega(K^{\xi}),\Omega(JK^{\xi})].
\end{align}
Evaluating the right hand side of \cref{laxgeometric} gives
\[[\Omega(K^{\xi}),\Omega(JK^{\xi})]=\begin{pmatrix} i(-h+2v_-)+j\frac{\dd}{\dd t}\sqrt{v_-} &\frac{\sqrt{h}}{4\sqrt{2}}\frac{\dd}{\dd t}\log(v_-v_+) \\ -\frac{\sqrt{h}}{4\sqrt{2}}\frac{\dd}{\dd t}\log(v_-v_+) &i(h-2v_+)+j\frac{\dd}{\dd t} \sqrt{v_+} \end{pmatrix}. \]
Hence, \cref{laxgeometric} is equivalent to \cref{odesystem}. In other words, we have found a Lax representation of the ODE system. This proves the following lemma.
\begin{lma}
 The eigenvalues of $\Omega(K^{\xi})\in \mathfrak{sp}(2)\subset \mathfrak{su}(4)$ are constant along $JK^{\xi}$.
\end{lma}
Introduce the variables
\[q_-=\frac{1}{2}\log(v_-),\quad r_-=\dot{q_-},\quad q_+=\frac{1}{2}\log(v_+),\quad r_+=\dot{q_+}\]
where the dot denotes the derivative with respect to $\frac{\dd}{\dd t}$, i.e. along $JK^{\xi}$.
Then \cref{odesystem} is equivalent to 
\begin{align}
\begin{split}
\label{odesystem3}
\dot q_-&=r_-, \hspace{1.5cm}
 \dot r_-=2(C^2 \exp(-(q_-+q_+))-2\exp(2q_-)) \\
 \dot q_+&=r_+, \hspace{1.5cm}
 \dot r_+=2(C^2 \exp(-(q_-+q_+))-2\exp(2q_+)).
\end{split}
\end{align}
This system is Hamiltonian with 
\[H=2(C^2\exp(-(q_-+q_+))+\exp(2q_-)+\exp(2q_+))+\frac{1}{2}r_-^2+\frac{1}{2}r_+^2.\] The Hamiltonian is in the form of \cite[Theorem 1]{bogoyavlensky1976perturbations}. In other words, \cref{odesystem3} are the equations for a generalised, periodic Toda lattice for the Lie algebra $\mathfrak{sp}(2)$. Bogoyavlensky's Lax representation for such a system coincides with \cref{laxgeometric}. Thus we have proven
\begin{prop}
 Simply connected, embedded, transverse $J$-holomorphic curves with a one-dimensional symmetry give rise to solutions to the 1D periodic Toda lattice equations for the Lie algebra $\mathfrak{sp}(2)$ with Lax representation
 \[ \frac{\dd}{\dd t}(\Omega(K^{\xi}))=[\Omega(K^{\xi}),\Omega(JK^{\xi})].\]
\end{prop}
Define 
\begin{align}
\label{H1 H2 definition}
\begin{split}
H_1&=2H \\
H_2&=2 C^2 e^{-q_--q_+} \left((r_--r_+)^2+4 \left(e^{q_-}+e^{q_+}\right)^2\right)+\left(r_-^2-r_+^2+4 e^{2 q_-}-4 e^{2 q_+}\right)^2.
\end{split}
\end{align}
Then one can check explicitly that the eigenvalues of $\Omega(K^{\xi})$ are given by 
\begin{align}
\label{H1 H2 eigenvalues}
\pm \frac{i}{2}\sqrt{H_1\pm \sqrt{H_2+16C^2\mathrm{Re}(\lambda^4)}}.
\end{align}
We have seen that a choice of a Killing vector field on $\cp$ gives rise to the above ODE system. The converse statement is also true.
\begin{prop}
 Let $X\cong \C$ be a Riemann surface equipped with coordinates $(u,t)$ a metric $k=C \gamma^2 (\dd u^2+\dd t^2)$ and ${\alpha_-},{\alpha_+}\colon X\to \R^{>0}$ satisfy the ODE system \cref{odesystem} for some $C>0$. Then there is an element $\xi\in \mathfrak{sp}(2)$ such that $X$ is an integral manifold of the distribution $V_{\xi}$ and such that ${\alpha_-},{\alpha_+}$ are the angle functions and the tautological embedding $X\to \cp$ is $J$-holomorphic and isometric.
\end{prop}
\begin{proof}
Let $({\alpha_-},{\alpha_+})$ be a solution of \cref{odesystem} with $C>0$. We will now give a $\xi \in \mathfrak{sp}(2)$ such that the integral submanifold through $[e]=[1,0,0,0]\in \cp$ has exactly the angle functions ${\alpha_-},{\alpha_+}$.
We have already seen that such integral submanifolds satisfy \cref{odesystem}. The system \cref{odesystem} itself does not depend on the choice of $\xi$. In fact, $\xi$ will determine its initial condition. Hence, it suffices to show that a $\xi$ with the desired properties can be found for any $C>0$ and fixed initial conditions for $\alpha_-(0),\alpha_+(0),\dot{\alpha}_-(0),\dot{\alpha}_+(0)$.
 Let $h_0=C{\alpha_-}(0)^{1/2}{\alpha_+}(0)^{1/2},{v_-}_0=C {\alpha_-}(0)^{3/2}{\alpha_+}(0)^{-1/2},{v_+}_0=C {\alpha_+}(0)^{3/2}{\alpha_-}(0)^{-1/2}$ and 
\[\xi=\begin{pmatrix} \frac{i}{8}(-3 \frac{\dot{\alpha}_-(0)}{{\alpha_-}(0)}+ \frac{\dot{\alpha}_+(0)}{{\alpha_+}(0)})+j\sqrt{{v_-}_0} & -\frac{h_0}{\sqrt{2}} \\ \frac{h_0}{\sqrt{2}} & \frac{i}{8}(- \frac{\dot{\alpha}_-(0)}{{\alpha_-}(0)}+3\frac{\dot{\alpha}_+(0)}{{\alpha_+}(0)})+j\sqrt{{v_+}_0}\end{pmatrix}.\]
Consider the integral manifold of $V_{\xi}$ passing through the point $[e]=[1,0,0,0]\in \cp$. By construction, since at the point $e\in \mathrm{Sp}(2)$ we have $\Omega_{MC}(K^{\xi})=\xi$ and so $e\in \mathrm{Sp}(2)$ lies in the reduced bundle $R_X$. Hence, we can apply \cref{summaryreducedbundle} and evaluate $\rho_1$ and $\rho_2$ at $K^{\xi}([e])$ to see that the angle functions ${\alpha_-},{\alpha_+}$ of $X$ satisfy the given initial conditions.
\end{proof}
\subsection{The $\mathbb{T}^2$-action}
\label{torussubsection}
In this subsection, we will investigate describe the geometry of a general $\uo$-action on $\cp$ with respect to the vector field $JK^{\xi}$ and make use of the fact that such an action commutes with a subgroup of $\mathrm{Sp}(2)$, which is generically a two-torus.
To that end, we fix 
\[\xi=\begin{pmatrix}
 ik & 0 \\ 0 & im \end{pmatrix} \in \mathfrak{sp}(2)
\]
which integrates to the $\mathrm{U}(1)$-action $\rho(e^{i\vartheta}[Z_0,Z_1,Z_2,Z_3])= [e^{k i\vartheta} Z_0,e^{-k i\vartheta} Z_1,e^{m i\vartheta}Z_2,e^{-m i\vartheta} Z_3]$ on $\cp$.
We fix an isomorphism $\mathfrak{sp}(2)\cong \mathfrak{so}(5)$ which maps $\xi$ to the element $(k+m,-k-m,k-m,m-k,0)\in \mathfrak{so}(5)$ and thus corresponds to the $\uo$-action with weights $(\tilde{k}=m-k,\tilde{m}=m+k)$ on $S^4\subset \C\oplus \C \oplus \R$.
The Lawson torus $T_{\tilde{k},\tilde{m}}$ and the bipolar torus $\tilde{T}_{\tilde{k},\tilde{m}}$ admit $J$-holomorphic twistor lifts which are invariant under $\rho$ and will be denoted by $\tau_{k,m}$ and $\tilde{\tau}_{k,m}$ respectively.
Let $\mathrm{Stab}(\xi)=\{g\in \mathrm{Sp}(2)\mid g^{-1} \xi g=\xi \}$, observe that
\[\mathrm{Stab}(\xi)\cong \begin{cases} \mathrm{U}(1)\times \mathrm{Sp}(1) \quad &\text{for } k=0 \text{ or } m=0 \\ \mathrm{U}(2) \quad &\text{for } k=\pm m \\ \mathbb{T}^2 \quad &\text{otherwise} \end{cases}.\]
We will refer to the first two cases as degenerate and to the last case as non-degenerate. Due to the presence of a larger symmetry group, the degenerate cases are simpler and have been treated by the author elsewhere. This is why, in this article, we restrict ourselves to the degenerate case in which $\mathrm{Stab}(\xi)$ is equal to the standard two-torus in $\mathrm{Sp}(2)$ i.e. $\{\mathrm{diag}( e^{i \theta}, e^{i\phi})\}$ to which we will simply refer as $\mathbb{T}^2$ unless stated otherwise. Let $\xi_1,\xi_2$ be the elements in $\mathfrak{sp}(2)$ corresponding to the action of $e^{i\theta}$ and $e^{i\phi}$ respectively.
This torus action gives rise to the multi-moment map $\nu=\omega(K^{\xi_1},K^{\xi_2})$, introduced in \cite{russo2019nearly}. 
\begin{prop}
\label{nk momentmap explicit}
The nearly Kähler multi-moment map on $\cp$ is given in homogeneous coordinates by
\[\nu=\frac{12}{|Z|^4}\mathrm{Im}(Z_0 Z_1 \overline{Z_2} \overline{Z_3}).\]
\end{prop}
There are two Clifford tori which arise as orbits of $\mathbb{T}^2$ and are equal to the preimages of the extremal values of $\nu$. Hence they are $\mathbb{T}^2$ orbits of the points $[1,1,1,i]$ and $[1,1,1,-i]$.
When investigating $\rho$-invariant $J$-holomorphic curves, it proves worthwhile to study lower-dimensional subsets $Y$ of $\cp$ which are both invariant under $\rho$ and the flow of $JK^{\xi}$. In other words, the distribution $V_{\xi}$ is then tangent to $Y$ and the problem of finding $\rho$-invariant $J$-holomorphic curves can be done separately on each such $Y$. Since $\dd \nu=\psi^+(K^{\xi_1},K^{\xi_2},\cdot)$, $V_{\xi}$ is tangent to each preimage $\nu^{-1}(c)$. The value $0$ as well as extrema of $\nu$ have a distinguished geometrical importance, there are the following sets which arise in a natural geometric way and to which $V_{\xi}$ is tangent
\begin{itemize}
\item $\mathcal{C}=\nu^{-1}(\{\nu_{\mathrm{min}}, \nu_{\mathrm{max}} \})$
 \item $\mathcal{B}=\nu^{-1}(0)$.
 \item $\mathcal{Q}$, the quadric associated to $\xi$ under the identification $\mathfrak{sp}(2)$ with the real part of $S^2(\C^4)$
 \item $\mathcal{S}$, the set where $\mathbb{T}^2$ does not act freely
 \item $\mathcal{T}$, a distinguished $S^1$ bundle over $S^3_0\subset \R^4\oplus \{0\}$ (for $(k,m)$ non-degenerate)
 \item $\tilde {\mathcal{T}}$, a four-dimensional submanifold constructed from $\mathcal{T}$ via Lawson's bipolar construction.
\end{itemize}
In the following we will define and outline the properties of each of the subsets. \par
The four points $\{[1,0,0,0],[0,1,0,0],[0,0,1,0],[0,0,0,1]\}$ are the fixed points of the $\mathbb{T}^2$-action. It turns out that $\mathbb{T}^2$ acts on the projective line going through any two of them with cohomogeneity one. More specifically, let
\begin{align*}
L_1&=\{Z_0=Z_1=0\},\quad L_2=\{Z_0=Z_2=0\}, \quad L_3=\{Z_0=Z_3=0\},\\
L_4&=\{Z_1=Z_3=0\} ,\quad L_5=\{Z_1=Z_2=0\}, \quad L_6=\{Z_2=Z_3=0\}.
\end{align*}
Observe that \[L_2\cup L_4={(K^{\xi_1})}^{-1}(0) \quad L_3 \cup L_5={(K^{\xi_2})}^{-1}(0), \quad (K^{\xi_1}+K^{\xi_2})^{-1}(0)=L_1, \quad (K^{\xi_1}-K^{\xi_2})^{-1}(0)=L_6.\] The action of $\mathbb{T}^2$ on $\cp$ is free away from the projective lines $L_1,\dots,L_6$, i.e. $\mathcal{S}=L_1 \cup \dots \cup L_6$.
Note that $L_1$ and $L_6$ are twistor lines, i.e. $h=0$, while $L_2,\dots,L_5$ are superminimal, i.e. $v_-=0$. Since $L_2,\dots,L_5$ project to totally geodesic two-spheres they furthermore satisfy $v_+=0$, see \cref{twistor second ff}.
Via the isomorphism $\mathfrak{sp}(2)\otimes \C \cong S^2(\C^4)$, $\xi$ is identified with the polynomial
\begin{align}
\label{quadrpolynomial}
i(kZ_0Z_1+mZ_2Z_3).
\end{align}
Furthermore, since $\mathfrak{sp}(2)\cong \mathfrak{so}(5)$, $\xi$ defines a vector field $K_{S^4}^\xi$ on $S^4$. Since it is a Killing vector field, $\nabla(K_{S^4}^\xi)$ can be identified with a two form on $S^4$. Its anti-self-dual part satisfies the twistor equation and gives thus rise to a quadric on $\cp=Z_-(S^4)$, holomorphic with respect to $J_1$, see \cite[ch. 13]{besse2007einstein}. This quadric is given by 
\[\mathcal{Q}=\{ k Z_0 Z_1+mZ_2Z_3=0\}\]
and coincides with the vanishing set of the quadratic expression \cref{quadrpolynomial}. 
For $x\in \cp$, let $X_x$ be the unique $\rho$-invariant embedded $J$-holomorphic curve containing $x$ and $\Sigma_x$ be the corresponding minimal surface in $S^4$. Define
\[\mathcal{T}=\{x\in \cp\mid \Sigma_x \text{ is contained in a totally geodesic } S^3 \}.\]
\begin{lma}
\label{tau degenerate}
 If $k=m$ then $\mathcal{T}=\cp$.
\end{lma}
\begin{proof}
 For $k=m=1$ the action on $S^4$ only rotates the first two components of $\R^2\oplus \R^3$. This action commutes with $\mathrm{SO}(2)\times \mathrm{SO}(3)$. Let $\Sigma$ be a minimal surface containing the orbit $\mathcal{O}_x$ for some $x$. Let $\nu$ be the normal bundle of $\mathcal{O}_x$ in $\Sigma$ and $v\in \nu_x$. By an action of $\mathrm{SO}(2)\times \mathrm{SO}(3)$ we can assume that $x$ and $v$ lie in $\R^4\oplus \{0\}$. In particular, $x$ is contained in $S_0^3\subset \R^4\oplus \{0\}$ and $v$ is tangent to it. Because $\nu$ and $S^3$ are $\rho$-invariant we have that $\mathcal{O}_x\subset S^3$ and $\nu_x \subset TS^3$. 
 By \cite{uhlenbeck1982equivariant}, there is an embedded, $\rho$-invariant minimal surface $\Sigma'$ such that $\mathcal{O}_x \subset \Sigma'\subset S^3$. Now, \cref{uniqueness u1 maps} implies that $\Sigma=\Sigma'$, i.e. $\Sigma\subset S^3_0$. 
\end{proof}
\begin{lma}
 If $k\neq m$ and $\Sigma$ is a $\rho$-invariant minimal surface which is contained in a totally geodesic $N\cong S^3$ then $N=S_0^3=\R^4\oplus \{0\} \cap S^4$. 
\end{lma}
\begin{proof}
 One can check that $\R^4\oplus \{0\} \cap S^4$ is the only totally geodesic $S^3$ on which $\rho$ acts. The statement follows from \cref{G-inv}.
\end{proof}
From now on, assume $k\neq m$. 
Over the totally geodesic $S^3_0$ we introduce the reduced Grassmannian bundle
\[\mathrm{Gr}(K_{S^4}^\xi,S_0^3)=\{V \in \widetilde{\mathrm{Gr}_2}(S^4) \mid K^\xi \in V \subset TS_0^3\}\] 
and notice that $\mathrm{Gr}(K^\xi,S_0^3)$ is in fact an $S^1$-bundle over $S_0^3$ which is mapped diffeomorphically to $\mathcal{T}\subset \cp=Z_-(S^4)$ by the projection $\widetilde{\mathrm{Gr}_2}(S^4)\to Z_-(S^4)$. Note that $\mathcal{T}$ is also invariant under $\rho$ and $JK^{\xi}$ is tangent to $\mathcal{T}$.
For a minimal surface $\Sigma$ denote by $\tilde{\Sigma}$ the bipolar surface from \cref{minimal s3}. Let 
\[\tilde{\mathcal{T}}=\{x \in \cp \mid \Sigma_x=\tilde S \text{ for $S$ a minimal surface lying in a totally geodesic $S^3$}\}.\]
For both $\mathcal{T}$ and $\tilde{\mathcal{T}}$ we obtain explicit expressions in \cref{computation section}.
\subsection{Separating $\mathbb{T}^2$-orbits}
As highlighted in the introduction, it is desirable to have a geometric construction of a map into $\R^4$ which descends to a local homeomorphism to the $\mathbb{T}^2$ quotient of $\cp$, at least away from a singular set. In general, one candidate for such a map is $(\nu,\|K^{\xi_1}\|, \|K^{\xi_2}\|,g(K^{\xi_1},K^{\xi_2}))$. For $S^3\times S^3$ such a map cannot have four-dimensional image however, due to the presence of a unit Killing vector field on $S^3\times S^3$. This is a special case, because $S^3\times S^3$ is the only nearly Kähler manifold admitting a unit Killing vector field \cite{moroianu2005unit}. Nevertheless, it seems difficult to compute the differentials of $g(K^{\xi_i},K^{\xi_j})$ in a general setting.\par
For the case $M=\cp$ we use a variant $p$ of this map, based on the splitting $T\cp=\mathcal{H}\oplus \mathcal{V}$. It turns out that, up to a sign, the multi-moment map $\nu$ can be expressed as a function of $\nu$, \cref{expression multimoment}.
The quotient $_{\mathbb{T}^2}\backslash^{\cp\setminus \{L_1\cup \dots \cup L_6\}}$ is smooth and will be denoted by $M_F$.
Note that the functions $(v_+,v_-,r_+,r_-)$ are all $\mathbb{T}^2$-invariant. 
\begin{thm}
\label{fibrationlemma}
 The functions \[p=(v_+,v_-,r_+,r_-)\colon M_F\to \R^4\] map to a bounded set $D\subset \R^4$ over which $p$ is a branched double cover. The two different points in the fibres of $p$ are complex conjugates of each other and the branch locus is equal to $\mathcal{B}=\nu^{-1}(0)$.
\end{thm}
\begin{proof}
Let 
\begin{align*}
\mathfrak{sp}(2)_F=\left\{\begin{pmatrix} q_1 && -\overline{q_3} \\ q_3 && q_2 \end{pmatrix} \in \mathfrak{sp}(2)\mid q_3\in \mathbb{H}\setminus\{0\},\quad q_1 \in \mathbb{H}\setminus\C \right\}
\end{align*}
and $\mathfrak{r}_F=\mathfrak{r}\cap \mathfrak{sp}(2)_F$.
The aim is to define a smooth map $M_F \to D\subset \R^6$ as a composition of 
\begin{align}
\label{mapcomposition}
p\colon M_F \overset{c_F}{\to} \mathfrak{sp}(2)_F/{S^1\times S^3} \overset{\Pi}{\to} \mathfrak{r}_F/{K_F} \overset{\bar{\zeta}}{\to} \bar{D}\subset \R^6 \overset{\mathrm{pr}}{\to} D\subset \R^4.
\end{align}
We will now define each of the maps individually and establish its properties.
\\
$c_F$:\\
Consider the map $c\colon \mathrm{Sp}(2)\to \mathfrak{sp}(2),\quad g\mapsto g^{-1} \xi g$.
Denote the image of $c$ by $\mathcal{O}^{\xi}$.
Note that $c$ descends to a diffeomorphism from $_{\mathbb{T}^2}\backslash^{\mathrm{Sp}(2)}$ onto its image. Quotienting both spaces by the right action of $S^1 \times S^3$ gives a homeomorphism between $_{\mathbb{T}^2}\backslash^{\cp}$ and $\mathcal{O}^{\xi}/{S^1\times S^3}$. Denote the restriction of this map to $M_F$ by $c_F$ and observe that $c_F$ maps $M_F$ diffeomorphically onto
\[c_F(M_F)=\mathcal{O}^{\xi}_F/{S^1\times S^3}\]
for 
$\mathcal{O}_{F}^{\xi}=\mathcal{O}^{\xi}\cap \mathfrak{sp}(2)_F$.
\\
$\Pi$:\\
If $x\in \mathfrak{sp}(2)_F$ then the orbit of $x$ under the action of $S^1\times S^3$ intersects $\mathfrak{r}_F$ in a $K_F$ orbit. In other words, there is an injective map $\Pi\colon \mathfrak{sp}(2)_F/{S^1\times S^3}\to \mathfrak{r}_F/{K_F}$. On the other hand, the inclusion $\mathfrak{r}_F\subset \mathfrak{sp}(2)_F$ induces $\iota\colon \mathfrak{r}_F/{K_F} \to \mathfrak{sp}(2)_F/{S^1\times S^3}$ and $\Pi\circ \iota=\mathrm{Id}$. This implies that $\Pi$ is in fact a a diffeomorphism from $\mathfrak{sp}(2)_F/{S^1\times S^3}$ to $\mathfrak{r}_F/{K_F}$. 
\\
$\bar{\zeta}$:\\
Consider the map 
\[\bar{\zeta}\colon \mathfrak{r}_F\to \R^6,\quad \begin{pmatrix} \frac{-i}{2} r_-+              
j \bar{\lambda} \sqrt{v_-} & -\frac{\overline{\lambda} \sqrt{h}}{\sqrt{2}} \\ \frac{\lambda \sqrt{h}}{\sqrt{2}} & \frac{i}{2}r_+ +j \lambda \sqrt{v_+} 
        \end{pmatrix}\mapsto (v_-,v_+,r_-,r_+,h,\mu=\mathrm{Re}(\lambda^4)).
 \]
which is $K_F$ invariant. %
On $\R_+^3\times \R^2 \times [-1,1]\ni (v_-,v_+,r_-,r_+,h,\mu)$, motivated by \cref{H1 H2 definition}, we define the functions 
\begin{align*}
 C^2&=h \sqrt{v_- v_+} \\ 
 H_1&=4h+4v_++4v_-+r_+^2+r_-^2 \\
 H_2&= ({r_-}^2 - {r_+}^2 + 4 v_- - 4 v_+)^2 + 8 h ((r_+ - r_-)^2 + 4 (v_+ + v_-)).
\end{align*}
Denote by $\bar{D}\subset \R^6$ the set defined by the equations
\begin{align}
\begin{split}
\label{eigenvaluerestrictions}
H_1&=4(k^2+m^2) \\
H_2+64C^2\mu&=16(k^2-m^2)^2.
\end{split}
\end{align}
The image of $\mathfrak{r}_F$ under $\bar{\zeta}$ is $\bar{D}$.
 This follows because \cref{eigenvaluerestrictions} describe how eigenvalues of elements in $\mathfrak{r}$ are calculated, i.e. \cref{H1 H2 eigenvalues}. Furthermore, since $\mathfrak{sp}(2)$ is semi-simple, conjugacy classes are uniquely characterised by their eigenvalues.
Note that $\bar{\zeta}$ descends to a double cover $\mathfrak{r}/{K_F}$, branched over $\R_+^3\times \R^2 \times \{-1,1\}$. The two preimages are obtained by switching between $\lambda$ and $\overline{\lambda}$. 
\\
$\mathrm{pr}$:\\
Since \cref{eigenvaluerestrictions} can be solved uniquely for $h$ and $\mu$ the projection $\mathrm{pr}$ from $\R_+^3\times \R^2 \times [0,1]$ to the first four components maps $\bar{D}$ diffeomorphically onto its image 
\begin{align}
\label{definition D}
\begin{split}
D=\{(v_-,v_+,r_-,r_+ \in (\R^{>0})^2\times \R^2 \mid &h=4(k^2+m^2)-4v_+-4v_--r_+^2-r_-^2 >0 \\
& \frac{16(k^2-m^2)^2-H_2}{64C^2} \in [-1,1] \}.
\end{split}
\end{align}
We have shown that \cref{mapcomposition} restricts to
\[M_F \overset{c_F}{\to} \mathcal{O}_{F}^{\xi}/{S^1\times S^3} \overset{\Pi}{\to} (\mathcal{O}^{\xi}\cap\mathfrak{r}_F)/{K_F} \overset{\bar{\zeta}}{\to} \bar{D}\subset \R^6 \overset{\mathrm{pr}}{\to} D\subset \R^4\]
and that $\bar{\zeta}$ is a branched double cover while each other map is a diffeomorphism. Hence, $p$ is a branched double cover. By \cref{killingeqn} and \cref{R lemma}, $p=(v_-,v_+,r_-,r_+)$. Observing that $p$ stays invariant under the map $\delta\colon [Z_0,Z_1,Z_2,Z_3]\mapsto [\overline{Z_0},\overline{Z_1},\overline{Z_2},\overline{Z_3}]$ completes the proof. \\
\end{proof}
The proof of \cref{fibrationlemma} gives an explicit description of the branch locus of $\zeta$.
The involution $\delta$ preserves the metric and reverses the almost complex structure on $\cp$. The fixed point set of $\delta$ is the standard $\mathbb{RP}^3$ which is hence a special Lagrangian submanifold for the nearly Kähler structure \cite{xu20063}.
\begin{lma}
 The set $\mathbb{RP}^3$ is special Lagrangian for the nearly Kähler structure on $\cp$.
\end{lma}
Given a nearly Kähler manifold with a $\mathbb{T}^2$-action the possibly singular space $_{\mathbb{T}^2}\backslash^{\nu^{-1}(0)}$ arises naturally. In the case of $\cp$ this space is an orbifold.
\begin{lma}
 The set $_{\mathbb{T}^2}\backslash^{\nu^{-1}(0)}$ is homeomorphic to $\mathbb{RP}^3/\{\pm 1\}$ where the action of $-1$ is given by 
 $[X_0,X_1,X_2,X_3]\mapsto [-X_0,-X_1,X_2,X_3].$
\end{lma}
\subsection{A torus fibration of $D$}
The image of $p$ is equal to $D$ and explicitly described by \cref{definition D}. To understand the set more conceptually, we show that $D$ itself admits a two-torus fibration over a rectangle $\mathcal{R}$ in $\R^2$. The torus fibres degenerate over the edges of the rectangle. 
Consider $(H_2,64C^2)$ as a map from $D$ to $\R^2$. Recall that $(ik,im)$ are purely imaginary and equal to
\[\frac{1}{2\sqrt{2}} \pm\sqrt{4(k^2+m^2) \pm \sqrt{H_2+64C^2\mu}}.\]
Since $\mu\in [-1,1]$ we obtain the inequalities 
\begin{align*}
H_2-64C^2\leq 16(k^2-m^2)^2 \leq H_2+64C^2, \quad H_2 \geq 64C^2, \quad H_2+64C^2 \leq 16(k^2+m^2)^2
\end{align*}
In fact, the image of $u=(H_2,64C^2)$ is equal to the rectangle, with one corner point missing
\begin{align*}
\mathcal{R}=\{(H_2,64C^2)\mid &H_2-64C^2\leq 16(k^2-m^2)^2 \leq H_2+64C^2 \\
&H_2 \geq 64C^2, \quad H_2+64C^2 \leq 16(k^2+m^2)^2\}\setminus\{(H_2,64C^2)=(16(k^2-m^2)^2,0)\}\subset \R^2.
\end{align*}
Note that $\mathcal{R}$ is bounded by the following line segments.
\begin{align*}
l_1&=\{H_2+64C^2=16(k^2-m^2)^2\}\cap \mathcal{R} \hspace{2cm}
l_2=\{H_2=16(k^2-m^2)^2+64C^2\}\cap \mathcal{R} \\
l_3&=\{H_2+64C^2=16(k^2+m^2)^2\}\cap \mathcal{R} \hspace{2cm}
l_4=\{64C^2=H_2\}\cap \mathcal{R}.
\end{align*}
\begin{figure}
\centering
\includegraphics[width=0.55\linewidth]{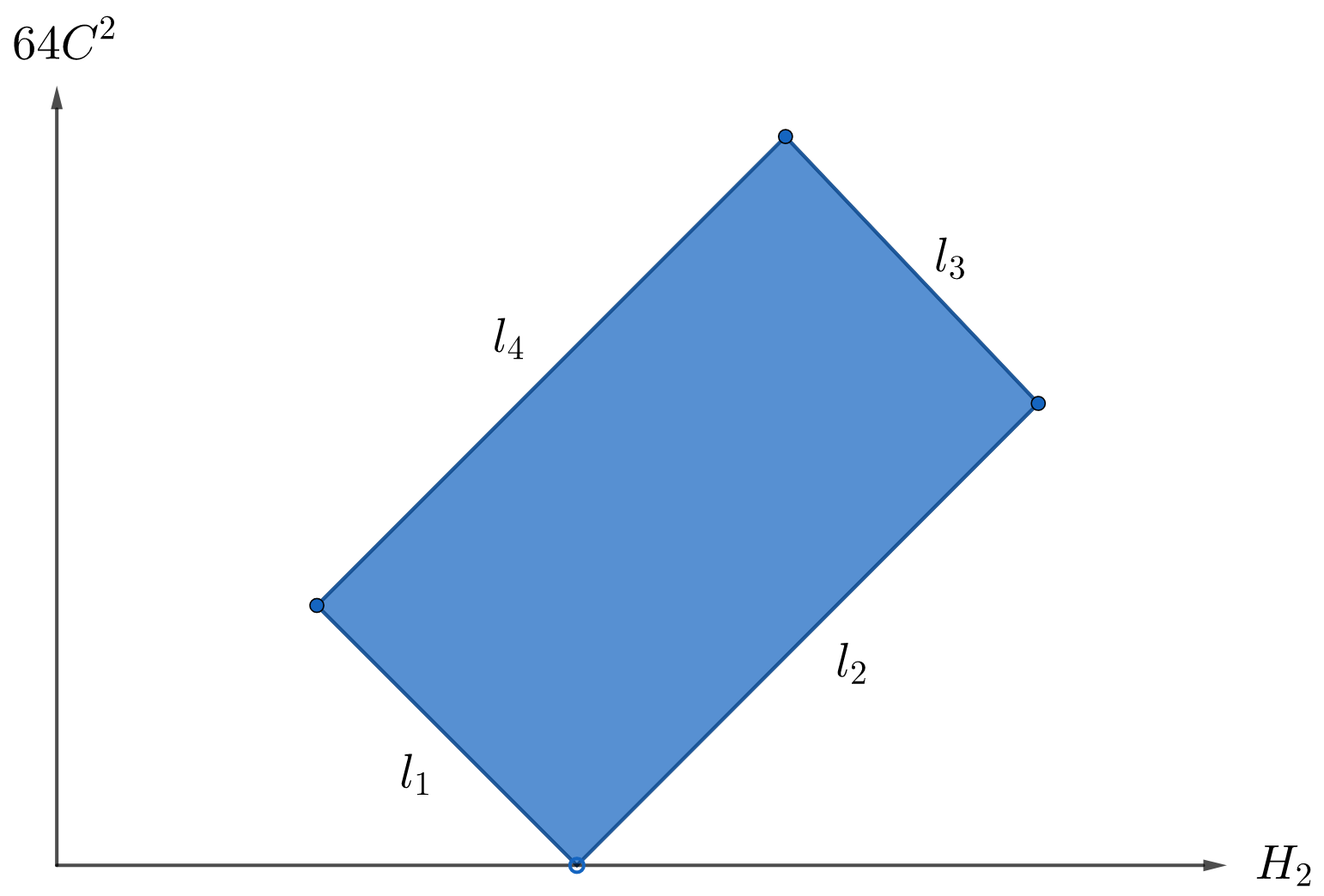}
\caption{The image of $u\colon D\to \R^2$. \cref{boundary R} expresses how the preimage of the boundary points is related to special subsets of $\cp$, compare with \cite[Fig. 1]{ferus1990s}.}
\label{figure R}
\end{figure}
Motivated by the Toda lattice treatment, define the symplectic form 
\[\dd q_- \wedge \dd r_-+ \dd q_+ \wedge \dd r_+=\frac{1}{2}(\dd (\mathrm{log}(v_-))\wedge \dd r_-+\dd(\mathrm{log}(v_+))\wedge \dd r_+).\]
Let $\mathcal{R}_0=\mathcal{R}\setminus\{l_1\cup \cdots \cup l_4\}$ the interior of $\mathcal{R}$ and $D_0=u^{-1}(\mathcal{R}_0)$. Using this symplectic form one can apply the Liouville-Arnold theorem to prove.
\begin{prop}
On $D_0$, $u$ is a submersion and its fibres are homeomorphic to $\mathbb{T}^2$. Around any $x\in D_0$ there are local coordinates 
\[I_1,I_2,\vartheta_1,\vartheta_2, \quad \vartheta_i\in \R/{\Z}\]
such that the angles $\vartheta_1,\vartheta_2$ are coordinates on $u^{-1}\{x\}$ and $I_1,I_2$ only depend on $H_2$ and $C^2$.
In these coordinates \cref{odesystem} is transformed to
\[\dot{I_i}=0, \quad \dot{\vartheta}=\omega_i(I_1,I_2,C^2)\quad i=1,2.\]
\end{prop}
While the previous proposition shows that fibres over $\mathcal{R}_0$ are regular it is to be expected that the preimages of the boundary of $\mathcal{R}$ correspond to subsets with special geometric properties. 
In \cref{torussubsection}, we have considered the distinguished subsets $\mathcal{S},\mathcal{T},\tilde{\mathcal{T}},\mathcal{C}$ and $\mathcal{B}$.
The following proposition is visualised in \cref{figure R} and summarises how the preimages of $l_1\cup\cdots \cup l_4$ are related to these subsets.
\begin{thm}
\label{boundary R}
 The special subsets $\mathcal{S},\mathcal{T},\tilde{\mathcal{T}},\mathcal{C}$ and $\mathcal{B}$ are related to the edges of $\mathcal{R}$ 
 \begin{itemize}
 \item $u^{-1}(l_2)=p(\mathcal{T})\cong \mathbb{D}^2$, see \cref{figure l2}
 \item $u^{-1}(l_3)=p(\widetilde{\mathcal{T}})\cong \mathbb{D}^2$
 \item $u^{-1}(l_1\cup l_4)=p(\mathcal{B}\setminus \mathcal{S})$
 \end{itemize}
while the corner points of $\mathcal{R}$ are related to the Clifford torus $\tau_{1,1}$, Lawson torus $\tau_{k,m}$ and the bipolar torus $\tilde{\tau}_{k,m}$ by 
\begin{itemize}
 \item $u^{-1}(l_3 \cap l_4)=p(\tilde{\tau}_{k,m})\cong S^1$
 \item $u^{-1}(l_1\cap l_2)=p(\tilde{\tau}_{k,m})\cong S^1$
 \item $u^{-1}(Q)=p(\mathcal{C})\cong\{*\}$.
\end{itemize}
In the case $k=m$ the rectangle $\mathcal{R}$ degenerates to the line $l_2$, which is a manifestation of \cref{tau degenerate}. If $m=0$ then $\mathcal{R}$ degenerates to the line $l_3$.
\end{thm}
Note that the map $u\circ p$ extends to the singular set, i.e. to the map \[P\colon \taction {\cp} \to \bar{\mathcal{R}}=\mathcal{R}\cup\{(H_2,64C^2)=(16(k^2-m^2)^2,0)\}\]
mentioned in the introduction and $P$ maps $\taction \mathcal{S}$ to the point $(16(k^2-m^2)^2,0)$. To prove \cref{boundary R} we first establish a more explicit understanding of the map $p$. 
\subsection{Relation between $p$ and moment maps} 
\label{computation section}
If one is just interested in a homeomorphism or cover $\tcp \to \R^4$ the functions $(v_-,v_+,r_-,r_+)$ are an unnecessarily complicated choice from a topological point of view. The reason for using these functions is that $JK^{\xi}$ takes a simple form
\[r_-v_-\frac{\partial}{\partial v_-}+r_+v_+\frac{\partial}{\partial v_+}+(h-2v_-)\frac{\partial}{\partial r_-}+(h-2v_+)\frac{\partial}{\partial r_+}.\]
To get an idea of what the functions $(v_-,v_+,r_-,r_+,h)$ look like in homogeneous coordinates
consider the following set of $\mathbb{T}^2$ invariant functions
\begin{align}
\begin{split}
 f_1&=|Z|^{-2}(|Z_0|^2+|Z_1|^2),\quad f_2=|Z|^{-2}(|Z_0|^2-|Z_1|^2) \\
 f_3&=|Z|^{-2}(|Z_2|^2+|Z_3|^2), \quad f_4=|Z|^{-2}(|Z_2|^2-|Z_3|^2) \\
 f_5&=|Z|^{-4}\mathrm{Re}(Z_0 Z_1 \overline{Z_2} \overline{Z_3}).
 \end{split}
\end{align}
Clearly, $1=f_1+f_3$ and the functions are also invariant under $\delta$. 
Note that the Kähler structure on $\cp$ admits a $\mathbb{T}^3$ action and after choosing an appropriate basis for ${\mathfrak{t}^3}^\vee$, $(f_1,f_2,f_4)$ are multiples of the symplectic moment map on $\cp$.
We can furthermore deduce the relation between the functions $f_1,\dots,f_5$ and $\nu$.
Note that $\nu$ is not invariant under $\delta$ but satisfies $\nu\circ \delta=-\nu$ and can thus not be expressed in terms of $f_1,\dots,f_5$. However, observe that the square of $\nu$ can be expressed in terms of the $f_i$ via
\begin{align}
\label{relation multimoment}
 (12\nu)^2+f_5^2=\frac{1}{16}(f_1^2-f_2^2)(f_3^2-f_4^2).
\end{align}
Denote by $D_f$ the image of $(f_1,f_2,f_4,f_5)$ in $\R^4$.
By writing down an explicit section one can prove that away from the branch locus $\mathcal{B}$ the $\mathbb{T}^2\times \{\pm 1\}$ principal bundle $(f_1,f_2,f_4,f_5)\colon \cp \to \R^4$ is trivial.
Our aim is to express $p=(v_-,v_+,r_-,r_+)$ in terms of the functions $f_1,\dots,f_5$. For a quaternion $q=z+jw\in \C\oplus j\C$ let $q_{\C}=z$ and $q_{j\C}=w$.
\begin{lma}
\label{conjugationlemma}
 $(v_-,r_-,v_+,r_+)=\zeta \circ \Pi\colon \mathfrak{sp}(2)_F \to D$ is given by
 \[\sigma=\begin{pmatrix} q_1 & -\overline{q_2} \\ q_2 & q_3 \end{pmatrix}\mapsto (2i (q_1)_{\C}, |(q_1)_{j\C}|^2,-2i {Q_3}_{\C}, |{Q_3}_{j\C}|^2)\]
  where $Q_3=q_2^{-1} q_3 q_2$.
\end{lma}
\begin{proof}
 The functions $(v_-,r_-,v_+,r_+)$ can be computed by conjugating $\sigma$ with an element in $S^1\times S^3$ to an element in $\mathfrak{r}$.
 Note that $\mathfrak{r}_F/K_F=\mathfrak{f}^1/K$ for 
 \[\mathfrak{f}= \left\{\begin{pmatrix} q_1 && -\overline{q_3} \\ q_3 && q_2 \end{pmatrix} \mid q_3 \in \C\right\} \subset \mathfrak{sp}(2), \quad \mathfrak{f}^1=\mathfrak{f}\cap \mathfrak{sp}(2)_F.\]
 Furthermore, $p$ is invariant under the group $K$ so it suffices to find $g\in \mathrm{Sp}(2)$ such that $g^{-1} \sigma g \in \mathfrak{f}^1.$
 Hence, we can conjugate $\sigma$ with the element $\mathrm{diag}(1,\frac{q_2}{|q_2|})$ and the result follows from the definition of $\bar{\zeta}$.
\end{proof}
We now pick a section \[\cp\setminus (L_1\cup L_6)\to \mathrm{Sp}(2),\quad [Z_0,Z_1,Z_2,Z_3]\mapsto \begin{pmatrix} h_1 & k_1 \\ h_2 & k_2 \end{pmatrix} \] with 
\begin{align*}
h_1&=|Z|^{-1}(Z_0+jZ_1),\quad h_2=|Z|^{-1}(Z_2+jZ_3) \\
k_1&=\frac{1}{f_1\sqrt{\frac{1}{f_1}+\frac{1}{f_3}}}(Z_0+jZ_1),\quad k_2=\frac{1}{f_3\sqrt{\frac{1}{f_1}+\frac{1}{f_3}}}(Z_2+jZ_3).
\end{align*}
Consequently, 
\begin{align}
\label{computec1}
c_F([Z_0,Z_1,Z_2,Z_3])=\begin{pmatrix} h_1 & k_1 \\ h_2 & k_2 \end{pmatrix}^{-1}\begin{pmatrix} ik & 0 \\ 0 & im \end{pmatrix} \begin{pmatrix} h_1 & k_1 \\ h_2 & k_2 \end{pmatrix}.
\end{align}
Let
\[E_{\pm}=r_{\pm}^2+4v_{\pm}.\] 
Combining \cref{computec1} with \cref{conjugationlemma} we can compute 
\begin{align}
 \begin{split}
 \label{coordinaterep}
 h&=-2 ((-1 + f_1) f_1 k^2 + 2 f_2 f_4 k m + (-1 + f_1) f_1 m^2 + 
  8 k m f_5) \\
 r_-&=-2 (f_4 k + f_2 m) \\
 E_-&=4 ((-1 + f_1)^2 k^2 + 2 f_2 f_4 k m + f_1^2 m^2 + 8 k m f_5) \\
 r_+&=-2 (f_4 k + f_2 m) - \frac{
 2 (k - m) (k + m) (f_1 f_4 k + (-1 + f_1) f_2 m)}{(-1 + f_1) f_1 k^2 + 
 2 f_2 f_4 k m + (-1 + f_1) f_1 m^2 + 8 k m f_5}\\
 E_+&=4 (f_1^2 k^2 + 2 f_2 f_4 k m + (-1 + f_1)^2 m^2 + 8 k m f_5).
 \end{split}
\end{align}
Denote by $\pi_{\mathbb{T}^2}$ the quotient map $\cp\setminus{\mathcal{S}}\to M_F$ by the $\mathbb{T}^2$-action. The equations \cref{coordinaterep} can in fact be solved for $(f_1,f_2,f_4,f_5)$.
This implies together with the triviality of the bundle $f\colon \mathbb{CP}^3\setminus {\mathcal{S}} \to D_f$ the following proposition.
\begin{prop}
\label{Df and D}
 There is a homeomorphism $D_f\to D$ which is a diffeomorphism in the smooth points such that the following diagram commutes
\[
\begin{tikzcd}
        & \mathbb{CP}^3\setminus {\mathcal{S}} \arrow[ld, "f"'] \arrow[rd, "p\circ \pi_{\mathbb{T}^2}"] &  \\
D_f \arrow[rr] &                                    & D
\end{tikzcd}.
\]
In particular $p\circ \pi_{\mathbb{T}^2}$ is a trivial $\mathbb{T}^2 \times \{\pm 1\}$ bundle when restricted to $\mathbb{CP}^3\setminus (\mathcal{S}\cup \mathcal{B})$. 
\end{prop}
Note that \cref{relation multimoment} yields the following corollary of \cref{Df and D}.
\begin{crl}
\label{expression multimoment}
 The square of the multi-moment map $\nu$ is a function of $(v_-,v_+,p_-,p_+)$
\end{crl}
The rest of the subsection consists of proofs of some of the statements of \cref{boundary R}.
Note that the preimage of $l_1,l_2$ in $\bar{D}$ are exactly all points with $\mu=1$ or $\mu=-1$ respectively. Hence, the preimage of $l_1\cup l_2$ is equal to the branch locus of $\zeta$.
We can check explicitly
\begin{align*}
 u^{-1}(l_4)&=\{v_-=v_+,\quad r_-=r_+\}, \\
 u^{-1}(l_3)&=\{(v_-,v_+,r_-,r_+)\in D\mid \sqrt{E_+}+\sqrt{E_-}=2\sqrt{l^2+m^2}, \quad r_+=-r_-\frac{\sqrt{E_+}}{\sqrt{E_-}}\}.
\end{align*}
 Note that on $u^{-1}(l_3)$ the equation $2h=\sqrt{E_+E_-}$ is satisfied.
\begin{lma}
 The set $\mathcal{T}$ can be explicitly described as $\{f_1=1/2, \quad f_4k=f_2m\}\subset \cp$.
\end{lma}
\begin{proof}
 This follows from \cref{coordinaterep} by
 setting $E_-=E_+$ and $r_-=r_+$.
\end{proof}
Note that on $D$ we have $C^2=(k^2+m^2-v_-v_+-1/4r_-^2-1/4 r_+^2)(\sqrt{vs})$. The maximum value of $C^2$ on $D$ is attained for $r_-=0,r_+=0$ and $v_-=v_+=\frac{1}{4}(k^2+m^2)$ resulting in $C^2_{\max}=\frac{1}{8}(k^2+m^2)^2$. Furthermore we have 
$h=\frac{1}{2}(k^2+m^2)$ which means ${\alpha_-}={\alpha_+}=\frac{1}{2}$ and is consistent with \cref{constantangle}. These solutions describe the two $\mathbb{T}^2$ invariant Clifford tori. In particular, if $k,m\neq 0$ then the $\mathbb{T}^2$ invariant tori never lie in $\mathcal{B}$.
Furthermore, we have $f_1=1/2, f_2=0,f_4=0,f_5=0$ and in particular $\frac{1}{16}(f_1^2-f_2^2)(f_3^2-f_4^2)-f_5^2=\frac{1}{256}$ which means $\nu=\pm 3/4$. That means that the two Clifford tori are equal to $\nu^{-1}(\pm 3/4)$ which are the extremal values of $\nu$. \par
The other special point on $l_4$ is the intersection point $l_1\cap l_4$. Its preimage under $p\circ u$ is identified with the subset of points in $[X_0,X_1,X_2,X_3]\in \mathbb{RP}^3/{\pm 1}$ that satisfy $f_1=1/2$ and $f_4k=f_2m$, i.e.
\[X_0^2+X_1^2=X_2^2+X_3^2, \quad k(X_2^2-X_3^2)=m(X_0^2-X_1^2).\]
Hence, ${p\circ u}^{-1}(l_1\cap l_4)$ and thus ${p\circ u}^{-1}(l_1\cap l_4)$ is homeomorphic to two copies of $\mathbb{RP}^1\cong S^1$.
\begin{lma}
 The preimage $u^{-1}(l_4)$ is homeomorphic to a closed two-disk in $D$.
\end{lma}
\begin{proof}
 Inserting $v_-=v_+$ and $r_-=r_+$ into the inequalities describing $\mathcal{R}$ impose
 \[(k^2-m^2)^2 \leq 8C^2 \leq (k^2+m^2)^2 \Leftrightarrow (k^2-m^2)^2 \leq 8 v_-(k^2+m^2-2v_--1/2r_-^2) \leq (k^2+m^2)^2 \]
 which is homeomorphic to a disk in the $(v_-,r_-)$ coordinates.
 \end{proof}
 The flowlines of $JK^\xi$ in $(p\circ u)^{-1}(l_4)$ are mapped to closed curves in $u^{-1}(l_4)$ which are plotted in \cref{figure l2}.
 \begin{figure}
\centering
\includegraphics[width=.35\linewidth]{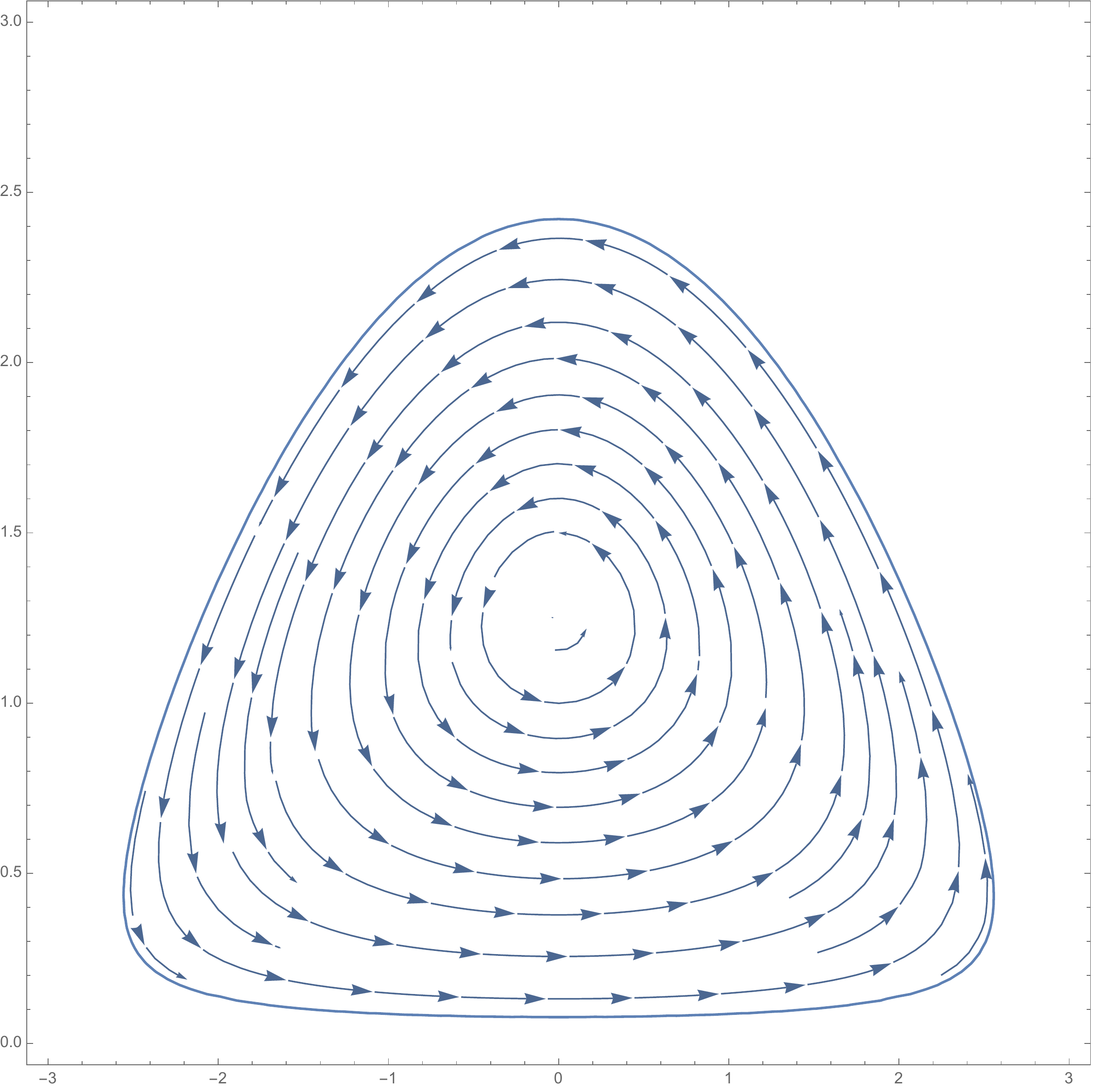}
\caption{Flow lines of $JK^{\xi}$ in $u^{-1}(l_4)$. The boundary of the region is equal to $u^{-1}(l_4\cap l_1)$ while the zero of $JK^{\xi}$ is the preimage of $u^{-1}(l_4\cap l_3)$.}
\label{figure l2}
\end{figure}
\subsection{Superminimal $\uo$-invariant curves}
The focus of this article has been on transverse $J$-holomorphic curves. But the computations carried out in \cref{computation section} give a framework for classifying $\uo$-invariant superminimal curves too.
From \cref{coordinaterep} we can deduce that $v_-=4|kZ_0Z_1+mZ_2Z_3|^2$
which implies 
\begin{lma}
\label{quadric superminimal}
 The vector field $K_x^{\xi}$ is horizontal in $x \in \cp$ if and only if $x$ lies on the quadric $\mathcal{Q}$.
\end{lma}
In the case $k=m$ this has already been observed in \cite[Corollary 6.3]{acharya2020circle} where this quadric has been described in more detail. Remarkably, $\mathcal{Q}$ is constructed from $\xi$ in three different ways, via $\mathfrak{sp}(2)\otimes \C \cong S^2(\C^4)$, via the anti-selfdual part of $\nabla(K^{\xi})$, as explained after \cref{quadrpolynomial}, and via \cref{quadric superminimal}.
If a $\rho$ invariant $J$ holomorphic curve intersects $\mathcal{Q}$ it will lie in $\mathcal{Q}$ entirely. In other words, $\mathcal{Q}$ is invariant under $\rho$ and $J K^{\xi}$ is tangent to $\mathcal{Q}$. Furthermore, it follows that $\mathcal{Q}$ is traced out by superminimal $\rho$-invariant $J$-holomorphic curves. The following proposition then follows from \cref{thmbryant}.
\begin{prop}
 All superminimal curves invariant under $\rho$ are given by
 \begin{align}
 &\varphi_C \colon \mathbb{CP}^1=\C \cup \{\infty\} \to \mathbb{CP}^3,\quad z\mapsto [1,\frac{C m}{m-k}z^{2k},z^{k-m},\frac{2 k C}{k-m}z^{k+m}] \\
 &\Psi_C \colon \mathbb{CP}^1=\C \cup \{\infty\}\to \cp,\quad z \mapsto [1,\frac{m}{m-k}z^2k,Cz^{k-m},\frac{k}{C(k-m)}z^{k+m}] \\
 &\text{or the projective lines } L_2,\dots,L_5 
 \end{align}
 for $C\in \mathbb{C}$. Furthermore, 
 \[\mathcal{Q}=\cup_{C\in \C} X_C\cup Y_C\cup L_2\cup \dots \cup L_5\]
 for $X_C=\mathrm{Im}(\varphi_C)\subset \cp$ and $Y_C=\mathrm{Im}(\psi_C)\subset \cp$.
\end{prop}
Observe that $X_L$ is totally geodesic if and only if $L=0$. This is equivalent to saying that $v_+$ vanishes on the curve as well. $X_0$ is part of the singular set of the torus action.
\printbibliography
\Addresses
\end{document}